\documentclass[12pt]{amsart}
\usepackage[T1]{fontenc}
\usepackage[utf8]{inputenc}
\usepackage{color,amssymb,amsmath,mathrsfs,enumerate}
\usepackage{hyperref}
\usepackage{tikz}

\newtheorem{theorem}{Theorem}
\newtheorem{lemma}[theorem]{Lemma}
\newtheorem{corollary}[theorem]{Corollary}
\newtheorem{prop}[theorem]{Proposition}

\theoremstyle{definition}
\newtheorem{definition}[theorem]{Definition}

\newtheorem*{acknowledgements*}{Acknowledgements}
\theoremstyle{remark}
\newtheorem{remark}[theorem]{Remark}
\newtheorem{note}[theorem]{Note}

\usepackage{color}

\newcommand{\R}{\mathbb R}

\newcommand{\T}{\mathbb T}

\newcommand{\D}{\mathbb D}
\def\he{{H}(X)}
\def\hf{H(\Lambda(\varphi))}
\def\laf{\Lambda(\varphi)}
\def\h{H^1(\D)}

\newcounter{obs}

      \renewcommand{\emptyset}{\varnothing}

\begin{document}

\title
[On extreme points of the unit ball]
{On extreme points of the unit ball\\ of a Hardy-Lorentz space}

\author{Sergey V. Astashkin}
\address{Department of Mathematics, Samara National Research University, Moskovskoye shosse 34, 443086, Samara, Russia; Lomonosov Moscow State University, Moscow, Russia; Moscow Center of Fundamental and Applied Mathematics, Moscow, Russia; Department of Mathematics, Bahcesehir University, 34353, Istanbul, Turkey.}
\email{astash56@mail.ru}
\urladdr{www.mathnet.ru/rus/person/8713}
\thanks{$^\dagger$\,This work was performed at Lomonosov Moscow State University  and was supported by the Russian Science Foundation, project no. 23-71-30001.}

\maketitle

\date{\today}


\begin{abstract}
We investigate the problem of a characterization of extreme points of the unit ball of a Hardy-Lorentz space $\hf$, posed by Semenov in 1978.
New necessary and sufficient conditions, under which a normalized function $f$ in $\hf$ belongs to this set, are found. The most complete results are obtained in the case when $f$ is the product of an outer analytic function and a Blaschke factor.

\end{abstract}

Primary classification: 46E30

Secondary classification(s): 30H10, 30J05, 46A55, 46B22 


\maketitle


\section{Introduction}
\label{S0}

Let $X$ be a Banach space. By ${\rm ball}\,(X)$ we denote the closed unit ball of  $X$, i.e., ${\rm ball}\,(X):=\{x\in X:\,\|x\|\le 1\}.$ An element $x_0\in {\rm ball}\,(X)$ is called an {\it extreme point} of this ball, if there is no non-degenerate segment contained in ${\rm ball}\,(X)$ for which $x_0$ is the midpoint.

Describing the set of extreme points of the unit ball of a Banach space is a classical problem in geometric theory of function spaces. In 1958, de Leeuw and Rudin gave the following remarkable characterization of extreme points of the unit ball
of the Hardy space $H^1:=\h$, where $\D:=\{z\in \mathbb{C}:\,|z|<1\}$ (see \cite[Theorem 1]{deleeuw-rudin} or \cite[Chapter~9]{hoffman}). 

{\bf Theorem A (de Leeuw-Rudin).} {\it
A function in $H^1$ is an extreme point of the unit ball of $H^1$ if 
and only if it is an outer function of norm one.
} 

The main motivation for this paper, as well as for the author's previous work \cite{A-prep}, was the open problem of characterization of extreme points of the unit ball of a Hardy-Lorentz space $\hf$. It was posed by Semenov in 1978 as a part of a collection of 99 problems in linear and complex analysis
from the  Leningrad branch of the Steklov Mathematical Institute  \cite[5.1, pp.23-24]{khavin-1978}. Later on, the same problem appeared in the subsequent list of problems in 1984 \cite[1.6, pp.22-23]{khavin-1984} and in 1994 \cite[1.5, p.12]{khavin-1994}. 

Note that the space $H^1$ belongs to the family of Hardy-Lorentz spaces, and
the desire to find out whether an extension of de Leeuw-Rudin theorem to Hardy-Lorentz spaces holds has been so undoubtedly one of the reasons to consider the above problem. Another reason was probably a simple structure of the set of extreme points of a Lorentz space (see Subsect. \ref{bryskin-sedaev}  below), which has been used often when estimating norms of operators in $L^p$ and, more generally, in symmetric spaces.

In 1974, de Leeuw-Rudin theorem was partially extended by Bryskin and Sedaev to more general Hardy-type spaces $\he$, where $X$ is a symmetric space on $[0,2\pi]$ (the preliminaries and, in particular, a more detailed account on Bryskin-Sedaev results can be found in the next section). Based on an inspection of the proof of de Leeuw-Rudin theorem, they proved that each normalized outer function from the space $\he$ is an extreme point of the unit ball of $\he$ whenever the norm of the underlined symmetric space $X$ is strictly monotone \cite[Theorem 1]{bryskin-sedaev} (in particular, this holds for a Hardy-Lorentz space $H(\Lambda(\varphi))$ if the function $\varphi$ is strictly increasing).
However, as was  proved in \cite{A-prep} (see Theorem 13), $H^1$ is a  unique space in this class such that the set of all extreme points  of its unit ball {\it coincides} with the set of all normalized outer functions.

Moreover, in \cite{A-prep}, the problem was resolved for the class of functions $f\in H(\Lambda(\varphi))$ such that the absolute value of the nontangential limit $f(e^{it}):=\lim_{z\to e^{it}}f(z)$ is constant on some subset of positive measure of the interval $[0,2\pi]$. Namely, such a  function of norm one turns to be an extreme point of the unit ball of ${\hf}$ if $\varphi$ is strictly increasing and strictly  concave on $[0,2\pi]$ \cite[Theorem 14]{A-prep}. Clearly,  this class contains all inner functions (see also the recent note \cite{CC}). At the same time, there are functions, satisfying the above condition, that are neither inner nor outer.

Thus, the set of extreme points of the ${\rm ball}\,(H(\Lambda(\varphi)))$ contains all normalized outer and inner functions whenever the function $\varphi$ is strictly increasing and strictly concave. Recalling that every $H^1$-function $f$ has the canonical representation $f=F\cdot I$, where $F$ is an outer, and $I$ is an inner functions, we see that the above problem reduces to investigation what products of inner and normalized outer factors are extreme points of the ${\rm ball}\,(H(\Lambda(\varphi)))$.

It is worth to mention here some recent results, closely related to de Leeuw-Rudin theorem and devoted to investigation of the structure of extreme points of the unit ball for other types of function spaces, which can be treated also as a generalization of the classical $H^1$. In particular, in the paper \cite{dyak},  a similar study was carried out for a punctured Hardy space $H^1_K$ consisting of  integrable functions $f$ on the unit circle whose Fourier coefficients $c_k(f)$ vanish either $k<0$ or $k\in K$ ($K$ is a fixed finite set of positive integers) (see also \cite{dyak2}).

Let us describe briefly the structure of the paper. In Section \ref{S1}, we gather some definitions and known results used in the sequel. In particular, in Subsect. \ref{Hardy-type spaces} and \ref{bryskin-sedaev} we recall the definition and properties of the Hardy-type spaces of analytic functions and especially of Hardy-Lorentz spaces as well as results related to the structure of the sets of extreme points of their unit balls.

Section \ref{S2} contains auxiliary results which we repeatedly use further. In Section \ref{S3}, we introduce some special multiplier classes, generated by $H^1$-functions. These classes are directly related to the description of extreme points of the unit ball of a Hardy-type space and have been  actually arisen implicitly already in the proof of Theorem A (see also \cite[Lemma~2.1]{dyak}).

In the next section, assuming that $\varphi$ is a strictly increasing and strictly  concave function on $[0,2\pi]$, we obtain necessary and sufficient conditions under which a function of norm one in the space ${\hf}$ is an extreme point of the unit ball of this space (see Theorem \ref{propos2}).
In  Theorem \ref{Th3a}, we specify this result by showing how the answer to the question of whether a given function is an extreme point of the ${\rm ball}\,(H(\Lambda(\varphi)))$ depends on the structure of the set of inner factors of $f$.

Section \ref{MAIN} contains the main results of the paper. First, we consider the situation when the decreasing rearrangement $\mu^*$ of the function $\mu:=|f(e^{it})|$ may be represented in the form $\mu^*=\mu(\omega)$, where $\omega:\,[0,2\pi]\to [0,2\pi]$ is a one-to-one measure-preserving mapping. Based on technical Proposition \ref{Th4a new}, we show that the set of "critical"\:(in some sense) points of the function $\mu^*$ satisfies certain conditions whenever the corresponding function $f$ is an extreme point of the unit ball of the space ${\hf}$
(see Theorems \ref{cor Th4a new} and \ref{Th4a general}).
As a consequence, in Corollaries \ref{ext of B-S-general} and \ref{ext of B-S}, we obtain new broad sufficient conditions imposed only on $\mu^*$, under which a normalized non outer function $f\in{\hf}$ is not an extreme point of the ${\rm ball}\,(\hf)$ (independently of its inner factors).

Finally, applying the results obtained, we give, in Theorem \ref{Th4a}, a criterion for a function $f$ of the form $f=F\cdot I_a$, where $F$ is an outer normalized function and $I_a$ is a Blaschke factor, to be an extreme point of the unit ball of the space ${\hf}$. These results show that geometric properties of such a function, under the additional assumption that the function $|\tilde{F}|$ decreases, depend, first of all, on the number and the structure of "critical"\:points of $|\tilde{F}|$ and, secondarily, on the choice of the point $a\in\D$ (see Corollary \ref{cor for Th4a}).



\section{Preliminaries.}
\label{S1}


\subsection{Symmetric spaces.}

A detailed exposition of the theory of symmetric spaces see in \cite{bennett-sharpley,krein-petunin-semenov,LT2}.

Let $T>0$. A Banach space $X:=X[0,T]$ of complex-valued Lebesgue-measurable functions on
the measure space $([0,T],m)$, where $m$ is the Lebesgue measure on the interval $[0,T]$, is called {\it symmetric} (or {\it rearrangement invariant}) if from the conditions $y \in X$ and
$x^*(t)\le y^*(t)$ almost everywhere (a.e.) on $[0,T]$ it follows that $x\in X$ and ${\|x\|}_X \le {\|y\|}_X.$ Here, $x^*(t)$ is the left-continuous, decreasing, nonnegative {\it rearrangement} of a measurable function $x$ defined by 
$$
x^{*}(t):=\inf \{ \tau\ge 0:\,m\{s\in [0,T]:\,|x(s)|>\tau\}<t \},\;\;0<t\le T.$$

Two measurable functions  $x$ and $y$ on $[0,T]$ are said to be {\it equimeasurable} whenever 
$$
m\{s\in [0,T]:\,|x(s)|>\tau\}=m\{s\in [0,T]:\,|y(s)|>\tau\},\;\;\tau>0.$$
In particular, a measurable function $x$ and its rearrangement $x^*$ are  equimeasurable. 
Clearly, if functions $x$ and $y$ are equimeasurable on $[0,T]$, then $x^*(t)=y^*(t)$, $0<t\le T$. Consequently, if $X$ is a symmetric space on $[0,T]$,
$x$ and $y$ are equimeasurable, $y\in X$, then $x\in X$ and $\|x\|_X=\|y\|_X.$

Recall that a mapping  $\omega:\,[0,T]\to [0,T]$ is called {\it measure-preserving}  if for every measurable set $E\subset [0,T]$ the inverse image $\omega^{-1}(E)$ is also measurable and $m(\omega^{-1}(E))=m(E)$. Given measure-preserving mapping $\omega$ and measurable on $[0,T]$ function $x(t)$, the function $x_\omega(t)=x(\omega(t))$ is also measurable on $[0,T]$ and $(x_\omega)^*(t)=x^*(t)$, $0<t\le T$. Therefore, if $X$ is a symmetric space, $x\in X$, then $x_\omega\in X$ and $\|x_\omega\|_X=\|x\|_X$ for such a mapping $\omega$.

Every symmetric space $X$ on $[0,T]$ satisfies  the continuous embeddings 
\begin{equation}
\label {embed1}
L^\infty[0,T]{\subset} X{\subset} L^1[0,T].
\end{equation}

If $X$ is a symmetric space on $[0,T]$, then the {\it K\"{o}the dual} 
(or {\it associated}) space $X'$ consists of all measurable functions $x$ on $[0,T]$ such that
$$
\|x\|_{X'}:=\sup\Bigl\{\int_0^T x(t)y(t)\,dt:\;\;
\|y\|_{X}\,\leq{1}\Bigr\}<\infty.
$$
One can easily check that $X'$ is also a symmetric space with respect to the norm $x\mapsto \|x\|_{X'}$, $X$ is continuously embedded into its second 
K\"{o}the dual $X''$ and $\|x\|_{X''}\le\|x\|_X$ for $x\in X$. 

A symmetric space $X$ {\it has the Fatou property} (or is {\it maximal}) if from $x_n\in X,$ $n=1,2,\dots,$ $\sup_{n=1,2,\dots}\|x_n\|_X<\infty$ and $x_n\to{x}$ a.e. on $[0,T]$ it follows that $x\in X$ and $||x||_X\le \liminf_{n\to\infty}{||x_n||_X}.$ Note that a symmetric space $X$ has the Fatou property if and only if $X''=X$ isometrically (see, e.g., \cite[Theorem~6.1.7]{KA}). 

The most familiar symmetric spaces are the $L^p$ spaces with the usual norm
\[
 \|x\|_{L^p}:=
 \left\{ 
 \begin{array}{ll}
    \Big(\int\limits_0^T|x(t)|^p\,dt\Big)^{1/p}, & 1\leq p<\infty\\
    \displaystyle
     \inf\limits_{m(E)=0}\sup\limits_{t\in [0, T]\setminus E} 
      |x(t)|, & p=\infty.
 \end{array}
 \right.
\]
Also, there are many other classes of symmetric spaces appearing in analysis; in particular, Orlicz, Lorentz, Marcinkiewicz spaces. Next, we will be interested primarily in Lorentz spaces, which are a natural generalization of the $L^1$ space.

Let $\varphi(t)$ be an increasing and concave function on $[0,T]$ such that $\varphi(0)=0$ and $\varphi(T)=1$. The {\it Lorentz space} $\Lambda(\varphi)$ is the set of all measurable functions $x$ on $[0,T]$, for which we have
\[
 \|x\|_{\Lambda(\varphi)}:=\int \limits_0^T  x^*(t)\,d\varphi(t)<\infty.
\]
In particular, if $\varphi(t)=T^{-1/p}t^{1/p}$, $1<p<\infty$, this space is denoted, as usual, by $L^{p,1}$ (see, e.g., \cite[\S\,4.4]{bennett-sharpley} or \cite[\S\,II.6.8]{krein-petunin-semenov}).  In the case when the function $\varphi$ is discontinuous at $0$ (resp. $\liminf\limits_{t\to 0} \varphi(t)/t<\infty$), up to equivalence of norms, we have: $\Lambda(\varphi)=L^\infty$ (resp. $\Lambda(\varphi)=L^1$).

The space $\Lambda(\varphi)$  has the Fatou property for every function $\varphi$.
If $\varphi(t)$ is continuous at $0$, then $\Lambda(\varphi)$ is separable and its K\"{o}the dual is the {\it Marcinkiewicz space} $M(\varphi)$ with the norm
\[
 \|x\|_{M(\varphi)}: =\sup\limits_{0<s\leq T} \frac1{\varphi(s)}\int\limits_0^s x^*(t)\,dt
\]
\cite[Theorem~II.5.2]{krein-petunin-semenov}.


 
\subsection{Hardy-type spaces of analytic functions.}
\label{Hardy-type spaces}

Let $X$ be a symmetric space with the Fatou property on the unit circle
$$
\T:=\{e^{it}:\, 0\le t<2\pi\},$$ 
which will be identified further with the interval $[0,2\pi)$.
The Hardy-type space $\he$ consists of all analytic functions $f\colon\D\to\mathbb{C}$, where $\D:=\{z\in \mathbb{C}:\,|z|<1\}$, such that
\begin{equation*}
\|f\|_{\he}:=\sup_{0\le r<1}\|f_r\|_X<\infty.
\end{equation*}
Here, $f_r(t):=f(re^{it})$ if $0\le r<1$ and $t\in[0,2\pi)$ (see \cite{bryskin-sedaev}). In particular, if $X=L^p$, $1\le p\le\infty$, we get the $H^p$-spaces, $H^p=H(L^p)$.  For matters related to this classical case we refer the reader to the books  \cite{duren,hoffman,koosis}.  

According to embeddings \eqref{embed1}, for every symmetric space $X$, we have $\he\subseteq H^1$. Consequently, for each function $f\in\he$ and almost all $t\in[0,2\pi)$ there exists the nontangential limit $f(e^{it}):=\lim_{z\to e^{it}}f(z)$. Moreover, precisely as in the classical $H^p$ case, the mapping  
\begin{equation}
\label{map} 
f\mapsto \tilde{f}:=f(e^{it})
\end{equation}
is an isometry between $\he$ and the subspace of $X$, consisting of all functions $g\in X$ such that the Fourier coefficients 
$$
c_k(g):=\frac{1}{2\pi}\int_0^{2\pi} g(t)e^{-ikt}\,dt,\;\;k\in\mathbb{Z},
$$
vanish if $k<0$.

Let us recall now the definition of inner and outer analitic functions. A function $I\in H^\infty$ is called {\it inner} if $|\tilde{I}(t)|=1$ a.e. on $[0,2\pi)$.

An important class of inner functions is formed by the so-called Blaschke products. For every $a\in\mathbb{D}$ we put
\begin{equation}
\label{factor}
I_a(z)=z\;\;\mbox{if}\;\;a=0,\;\;\mbox{and}\;\;I_a(z):=\frac{|a|}{a}\cdot\frac{a-z}{1-\bar{a}z}\;\;\mbox{if}\;\;a\ne 0.
\end{equation}
If $J\subset\mathbb{N}$ and $\{a_i\}_{i\in J}$ is a subset of the unit disk $\mathbb{D}$ satisfying the condition
\begin{equation*}
\sum_{i\in J} (1-|a_i|)<\infty,
\end{equation*}
the product $B_J:=\prod_{i\in J} I_{a_i}$ (called the {\it Blaschke product}) converges in $\mathbb{D}$, $B_J$ is an analytic function in $\mathbb{D}$, $|B_J(z)|\le 1$, $z\in \mathbb{D}$, and $|\tilde{B}_{J}(t)|=1$ for almost all $t\in[0,2\pi)$.

A non-null function $F\in H^1$ is termed {\it outer} if
$$
\ln |F(0)|=\frac{1}{2\pi}\int_{0}^{2\pi}\ln |\tilde{F}(s)|\,ds.$$
An outer function $F$ from $H^1(\D)$ can also be characterized by the following extremality property: if $g\in H^1(\D)$ and $\tilde{g}(t)\le \tilde{F}(t)$, $t\in [0,2\pi)$, then $|g(z)|\le |F(z)|$ for all $z\in\D$ (see, e.g., \cite[chapter~9]{hoffman}).

It is well known the following canonical (inner-outer) factorization theorem: every function $f\in H^1$, $f\not\equiv 0$, can be uniquely (up to a factor equal to one in modulus) represented as the product $f=IF$, where $I$ is an inner and $F$ is an outer functions (see, for  instance, \cite[\S\,IV.D.$4^0$]{koosis}).


\subsection{Extreme points and some related definitions and results.}
\label{bryskin-sedaev}

Let $X$ be a Banach space. As was mentioned above, a point $x_0\in {\rm ball}\,(X)$ is said to be {\it extreme} for the ${\rm ball}\,(X)$ if $x_0$ is not the midpoint of any non-degenerate segment contained in ${\rm ball}\,(X)$.  

Clearly, every extreme point $x_0$ of ${\rm ball}\,(X)$ is an element of norm one. Moreover, $x_0$ is an extreme point of ${\rm ball}\,(X)$ if and only if from $y\in X$ and $\|x\pm y\|\le 1$ it follows that $y=0$. 


As was said above, we will be interested in the structure of the set of extreme points of the unit ball of a Hardy-Lorentz space $H(\Lambda(\varphi))$. Therefore, it is worth to mention that analogous set for a Lorentz space $\Lambda(\varphi):=\Lambda(\varphi)[0,T]$ is contained in the set of functions 
$$f=\frac{\epsilon(t)}{\varphi(m(E))}\chi_E(t),$$
where $E$ is a subset of $[0,T]$ of positive measure and a measurable function $\epsilon:\,[0,T]\to \mathbb{C}$ satisfies the condition: $|\epsilon(t)|=1$ a.e. on $E$. Moreover, if $\varphi$ is strictly concave on $[0,T]$, then these two sets coincide (see  \cite{carothers-turett} and \cite{carothers-haydon-lin}).
Since ${\hf}$ is isometric to a subspace of $\Lambda(\varphi)$, from this result it follows immediately that every inner  function $f$ is an extreme point of the ${\rm ball}\,(H(\Lambda(\varphi)))$ (because $|\tilde{f}|=1$ a.e. on $[0,2\pi]$ and hence $\tilde{f}$ is an extreme point of the ${\rm ball}\,(\Lambda(\varphi))$).

Recall that a real-valued function $\psi$ defined  on $[0,T]$ is  called {\it strictly concave} if 
$$
\psi((1-\alpha)t_1+\alpha t_2)>(1-\alpha)\psi(t_1)+\alpha \psi(t_2)$$
for any $\alpha\in (0,1)$ and $0\le t_1,t_2\le T$, $t_1\ne t_2$.

An inspection of the proof of de Leeuw-Rudin theorem (see Lemma 1 in \cite{bryskin-sedaev} or Lemma \ref{l0} below) leads to the following partial extension of Theorem A to some Hardy-type spaces obtained by Bryskin and Sedaev. A symmetric space $X$ on $[0,T]$ is said to have \textit{a strictly monotone norm} if from the conditions: $x,y\in X$, $|x(t)|\le|y(t)|$ a.e. on $[0,T]$ and%
\begin{equation*}\label{1}
m\{t\in [0,T]:\,|x(t)|<|y(t)|\}>0
\end{equation*}
it follows that $\|x\|_X<\|y\|_X$.

{\bf Theorem B \cite[Theorem 1]{bryskin-sedaev}.} {\it If a symmetric space $X$ on $[0,2\pi]$ has a strictly monotone norm, then each normalized outer function from the space $\he$ is an extreme point of the ${\rm ball}\,({\he})$.}

As is easy to verify, a Lorentz space $\Lambda(\varphi)$ on $[0,2\pi]$ has the strictly monotone norm if and only if the function $\varphi$ is strictly increasing on $[0,2\pi]$. Thus, we obtain the following consequence of Theorem B.

{\bf Theorem C \cite[Corollary]{bryskin-sedaev}.} {\it
Let $\varphi$ be a strictly increasing, concave function on $[0,2\pi]$,
$\varphi(0)=0$, $\varphi(2\pi)=1$. Then every normalized outer function in $\hf$ is an extreme point of the ${\rm ball}\,({\hf})$.}

Thus, if $\varphi$ is a strictly increasing and strictly concave function on $[0,2\pi]$), the set of extreme points of the unit ball of a space $\hf$ is rather rich, because it contains all inner and all normalized outer functions. Moreover, as was above-mentioned, there are elements of this set that are neither inner nor outer functions (see \cite{A-prep}). 

However, 
according to the following Bryskin-Sedaev result, not all normalized functions in $\hf$ are extreme points of the unit ball of this space.

{\bf Theorem D \cite[Theorem 2]{bryskin-sedaev}.} {\it
Suppose that $\|f\|_{\hf}=1$ and $f(a)=0$ for some point $a\in\D$. If the function $\mu(t):=|\tilde{f}(t)|$ is continuously differentiable on $[0,2\pi)$ and for some $\delta>0$ the inequality $\mu^\prime(t) \leq -\delta$ holds for all $t\in [0,2\pi)$, then $f$ is not an extreme point of the ${\rm ball}\,({\hf})$.}

Further, in Corollaries \ref{ext of B-S-general} and \ref{ext of B-S} (see also Remark \ref{ext of B-S result}), we get a far-reaching strengthening of Theorem D.

We say that a function $\psi:\, [0,T]\to \mathbb{R}$ is {\it increasing} (resp. {\it strictly increasing}) if from $0\le t_1<t_2\le T$ it follows that $\psi(t_1)\le \psi(t_2)$ (resp. $\psi(t_1)< \psi(t_2)$). A function $\psi:\,[0,T]\to\mathbb{R}$ that becomes  (strictly) increasing after changing on a subset of $[0,T]$ of measure zero will be called (strictly) increasing as well.  A (strictly) decreasing function is defined similarly. 


In what follows, we assume that $\varphi$ is a concave, increasing and continuous function on $[0,T]$ such that $\varphi(0)=0$, $\varphi(T)=1$ (as a rule, $T=2\pi$). Finally, by $\Re\, (w)$ and $\Im\,(w)$ we denote the real and imaginary part of a complex number $w$, respectively.

\vskip 0.4cm


\section{Auxiliary results}
\label{S2}



As was already observed, the original proof of Theorem A (see \cite{deleeuw-rudin})  contains in fact the following useful fact (see Lemma 1 in \cite{bryskin-sedaev}).



\begin{lemma}
\label{l0}
Suppose the norm of a symmetric space $X$ on $[0,T]$ is strictly monotone. If $\|f\|_X=1$, $g\in X$, $m\{t:\,f(t)=0\}=0$, and 
$\|f\pm g\|_X\le 1$, then ${g}=h{f}$, where $h$ is a real-valued function, $|h|\le 1$ a.e. on $[0,T]$.
\end{lemma}

A proof of the following result can be found in \cite[Lemma 6]{CC} (for the special case of the spaces $L^{p,1}$,
$1<p<\infty$, see also   \cite[Lemma 1]{carothers-turett}). 

\begin{lemma}
\label{prop2}
Let $\varphi$ be a strictly increasing and  strictly concave  
function on $[0,T]$. The functions  $f,g \in \Lambda(\varphi)$ satisfy the equality 
$$\|f\|_{\Lambda(\varphi)} + \|g\|_{\Lambda(\varphi)}=\|f+g\|_{\Lambda(\varphi)}
$$
if and only if 
$$ f^* + g^*=(f+g)^*\;\;\mbox{on}\;[0,T]. 
$$
\end{lemma}

As is known (see, for instance, \cite[\S\,II.2, the proof of Property $7^0$]{krein-petunin-semenov}, for every function $x\in L_1[0,T]$ there exists a family of measurable subsets $\{E_t(x)\}_{0<t\le T}$ of the interval $[0,T]$ such that $m(E_t(x))=t$, $0<t\le T$, $E_{t_1}(x)\subset E_{t_2}(x)$ if $0<t_1<t_2\le T$, the embeddings   
\begin{equation}
\label{pra1}
\{s\in [0,T]:\,|x(s)|>x^*(t)\}\subset E_t(x)\subset \{s\in [0,T]:\,|x(s)|\ge x^*(t)\},\;\;0<t\le T,
\end{equation}
hold, and
\begin{equation}
\label{pra2}
\int_{E_t(x)}|x(s)|\,ds=\int_0^t x^*(s)\,ds,\;\;0<t\le T.
\end{equation}
Note that, by \eqref{pra2}, the function $F(t):=\int_{E_t(x)}|x(s)|\,ds$
is differentiable a.e. on $[0,T]$ and $F'(t)=x^*(t)$.

The following assertion is proved in \cite[Proposition~7]{A-prep}.

\begin{prop}
\label{propos}
Let $\varphi$ be a strictly increasing and  strictly concave  
function on $[0,T]$. Suppose  $f,g \in \Lambda(\varphi)[0,T]$, $f$ is nonnegative, $g$  is real-valued, $\|f\|_{\Lambda(\varphi)}=1$ and $\|f\pm g\|_{\Lambda(\varphi)}\le 1$. 

Then both functions $f\pm g$  are nonnegative a.e. on the interval $[0,T]$ and the sets $E_t(f)$ and $E_t(f\pm g)$, $0<t\le T$, can be chosen in such a way that
$$
E_t(f+g)=E_t(f-g)=E_t(f),\;\;0<t\le T.$$
Therefore, if $E_t:=E_t(f)$, $0<t\le T$, then simultaneously
\begin{equation}
\label{pr1}
\int_{E_t}f(s)\,ds=\int_0^t f^*(s)\,ds,\;\;0<t\le T,
\end{equation}
and
\begin{equation}
\label{pr2-e}
\int_{E_t}(f(s)\pm g(s))\,ds=\int_0^t (f\pm g)^*(s)\,ds,\;\;0<t\le T.
\end{equation}
\end{prop}

Let $g$ be a measurable nonnegative function on $[0,T]$. By Ryff's theorem (see, e.g., \cite[Theorem~2.7.5]{bennett-sharpley}), there exists a measure-preserving mapping $\tau:\,[0,T]\to [0,T]$ such that $g(t)=g^*(\tau(t))$ for a.e. $t\in[0,T]$, where $g^*$ is the decreasing rearrangement of $g$.
Sometimes, conversely, $g^*$ can be expressed as a  composition of $g$ and some measure-preserving mapping\footnote{In general, this is not true (see \cite[Example~2.7.7]{bennett-sharpley}).}. Assume, for instance, that $g$ (or its restriction to some subset of full measure of the interval $[0,T]$) is injective. Then, if $g(t)=g^*(\tau(t))$, where $\tau$ is a measure-preserving mapping of $[0,T]$, $\tau$ may be assumed to be one-to-one. Hence, there exists the inverse mapping $\tau^{-1}:\,[0,T]\to [0,T]$, which also preserves the measure, and  $g^*(t)=g(\tau^{-1}(t))$, $0\le t\le T$.

\begin{prop}
\label{reduction4}
Let $g$ and $h$ be two real-valued measurable functions on $[0,T]$, $g(t)> 0$, $|h(t)|\le 1$.
Suppose that there exists a one-to-one measure-preserving mapping $\omega:\,[0,T]\to [0,T]$ such that $g^*=g(\omega)$. Denote $h_\omega(t):=h(\omega(t))$, $0\le t\le T$.

The following conditions are equivalent:

(i) the sets $E_t(g)$ and $E_t(g\cdot(1\pm h))$, $0<t\le T$, can be chosen in such a way that
\begin{equation}
\label{lem eq}
E_t(g\cdot(1+ h))=E_t(g\cdot(1- h))=E_t(g),\;\;0<t\le T;
\end{equation}

(ii) for almost all $t\in[0,T]$
\begin{equation}
\label{lem eq-dop}
(g(1\pm h))^*(t)=g^*(t)\cdot(1\pm h_\omega(t));
\end{equation}

(iii) the functions $g^*(t)\cdot (1+h_\omega(t))$ and $g^*(t)\cdot (1-h_\omega(t))$ decrease on $[0,T]$.

In addition, we have
\begin{equation}
\label{lem eq2}
\Big(\int_{E_t(g)}g(s)h(s)\,ds\Big)_t'=g^*(t)h_\omega(t)\;\;\mbox{for almost all}\;t\in [0,T].\end{equation}
\end{prop}
\begin{proof}
First, since the mapping $\omega^{-1}:\,[0,T]\to [0,T]$ also preserves the measure and $g(\omega)=g^*$, 
$$
\int_{\omega((0,t))}g(s)\,ds=\int_0^t g(\omega(u))\,du=\int_0^tg^*(u)\,du,\;\;t\in[0,T].$$
Therefore, we can assume that $E_t(g)=\omega((0,t))$, $0<t\le T$, and hence,
\begin{eqnarray}
\int_{E_t(g)}g(s)\cdot(1\pm h(s))\,ds&=&\int_{0}^t g(\omega(s))(1\pm h(\omega(s)))\,ds\nonumber\\
&=&\int_{0}^t g^*(s)\cdot(1\pm h_\omega(s))\,ds.
\label{help h1}
\end{eqnarray}
Since
$$
\int_{E_t(g)} g(s)\,ds=\int_0^t g^*(s)\,ds,$$
then from \eqref{help h1} it follows that
$$
\int_{E_t(g)} g(s)h(s)\,ds=\int_0^t g^*(s)h_\omega(s)\,ds.
$$
Differentiating this relation, we obtain \eqref{lem eq2}.

We now proceed with the proof of the equivalence of conditions (i), (ii) and (iii).

$(i)\Longrightarrow (ii)$.
Applying \eqref{lem eq}, we obtain
$$
\int_{E_t(g)}g(s)\cdot(1\pm h(s))\,ds=\int_{E_t(g(1\pm h))}g(s)\cdot(1\pm h(s))\,ds=\int_{0}^t (g(1\pm h))^*(s)\,ds,$$
whence, by \eqref{help h1},
$$
\int_{0}^t g^*(s)\cdot(1\pm h_\omega(s))\,ds=\int_{0}^t (g(1\pm h))^*(s)\,ds.$$
Differentiating this equality, we obtain \eqref{lem eq-dop}.

Implication $(ii)\Longrightarrow (iii)$ is obvious.

$(iii)\Longrightarrow (i)$.
Since the mapping $\omega:\,[0,T]\to [0,T]$ preserves the measure, the functions $g\cdot(1\pm h)$ and $g(\omega)\cdot(1\pm h(\omega))$ are equimeasurable. Therefore, from the equality
$$
g(\omega(t))\cdot(1\pm h(\omega(t)))=g^*(t)\cdot(1\pm h_\omega(t)),\;\;t\in [0,T],$$
and the assumption it follows that 
$$
(g(1\pm h))^*(t)=g^*(t)\cdot(1\pm h_\omega(t))$$
for almost all $t\in [0,T]$.
Thus, 
\begin{eqnarray}
\label{lem eq2A}
(g(1+h))^*(t)+(g(1-h))^*(t)&=&
2g^*(t)\nonumber\\&=&(g(1+h)+g(1-h))^*(t),\;\;t\in [0,T],
\label{lem eq2A}
\end{eqnarray}
and, applying \cite[Chapter~II, Property $9^0$, p. 91]{krein-petunin-semenov}, we arrive at the relation
$$
E_t(g\cdot(1+ h))=E_t(g\cdot(1- h)),\;\;t\in[0,T].$$
Consequently,
$$
\int_{E_t(g\cdot(1+ h))}g(s)(1+ h(s))\,ds=\int_0^t(g(1+ h))^*(s)\,ds$$
and
$$
\int_{E_t(g\cdot(1+ h))}g(s)(1- h(s))\,ds=\int_0^t(g(1- h))^*(s)\,ds.$$
Summing these equalities, due to \eqref{lem eq2A}, we obtain
$$
2\int_{E_t(g\cdot(1+ h))}g(s)\,ds=\int_0^t((g(1+h))^*(s)+(g(1-h))^*(s))\,ds=2\int_0^t g^*(s)\,ds.$$
Thus, we can take for $E_t(g)$, $0<t\le T$, the sets $E_t(g\cdot(1+ h))$. As a result, \eqref{lem eq} is proved.
\end{proof}

Let $a\in\mathbb{D}$. It is well known that the Blaschke factor $I_a(z)$, $z\in\D$, defined by formula \eqref{factor}, is a linear-fractional isomorphism of the closed unit disk $\bar{\D}$ and also it is one-to-one, continuous mapping of   the unit circle $\T$ onto itself. Therefore, the function $t\mapsto \tilde{I_a}(t)$ maps continuously and one-to-one the interval $[0,2\pi)$ onto $\T$. In what follows, we will assume that the values of the argument are chosen in such a way that the function $t\mapsto {\rm{arg}} \tilde{I_a}(t)$ is continuous on $\mathbb{R}$. In particular, ${\rm{arg}} \tilde{I_a}(2\pi)={\rm{arg}} \tilde{I_a}(0)+2\pi$ or ${\rm{arg}} \tilde{I_a}(2\pi)={\rm{arg}} \tilde{I_a}(0)-2\pi$.


Further, we will repeatedly use the following technical lemma.

\begin{lemma}
\label{lemma: calcul}
Let $a\in\mathbb{D}$. There exist positive constants $c_a$ and $C_a$ such that
$$
|\sin(({\rm{arg}} \tilde{I_a}(v)-{\rm{arg}} \tilde{I_a}(u))/2)|\le C_a|u-v|\;\;\mbox{for any}\;u,v\in \mathbb{R}$$
and
$$
|\sin(({\rm{arg}} \tilde{I_a}(v)-{\rm{arg}} \tilde{I_a}(u))/2)|\ge c_a|u-v|\;\;\mbox{if}\;|u-v|\le\pi.$$
\end{lemma}
\begin{proof}
For all $z_1,z_2\in\mathbb{C}$ such that $|z_1|=|z_2|=1$ it holds
\begin{equation} \label{gen elem}
|\sin(({\rm{arg}} z_1-{\rm{arg}}z_2)/2)|=\frac12 |z_1- z_2|.
\end{equation}
Indeed\footnote{Equality \eqref{gen elem} easily follows also from elementary geometric considerations.}, for arbitrary $\theta,\eta\in \mathbb{R}$
\begin{eqnarray*}
|e^{i\eta}-e^{i\theta}| &=&|\cos \eta-\cos \theta+i (\sin \eta-\sin \theta)|\\&=&
\sqrt{2}(1-(\cos \eta\cos \theta+\sin \eta\sin \theta))^{1/2}\\&=&
\sqrt{2}(1-\cos (\eta-\theta))^{1/2}=2|\sin((\eta-\theta)/2)|.
\end{eqnarray*}
Hence, setting $z_1=e^{i\eta}$, $z_2=e^{i\theta}$, we obtain \eqref{gen elem}.

In particular,
by \eqref{gen elem}
\begin{equation} \label{gen elem1}
|\sin(({\rm{arg}} \tilde{I_a}(v)-{\rm{arg}} \tilde{I_a}(u))/2)|=\frac12 |\tilde{I_a}(v)- \tilde{I_a}(u)|.
\end{equation}

On the other hand, for any $u,v\in \mathbb{R}$ we have
\begin{eqnarray*}
|\tilde{I_a}(v)- \tilde{I_a}(u)| &=&\left|\frac{e^{iv}-a}{1-\bar{a}e^{iv}}-\frac{e^{iu}-a}{1-\bar{a}e^{iu}}\right|\\&=&
\frac{(1-|a|^2)|e^{iv}-e^{iu}|}{|1-\bar{a}e^{iv}||1-\bar{a}e^{iu}|},
\end{eqnarray*}
whence, in view of the inequality $|a|<1$, it follows:
$$
c_a'|e^{iv}-e^{iu}|\le |\tilde{I_a}(v)- \tilde{I_a}(u)|\le C_a'|e^{iv}-e^{iu}|,$$
where the constants $c_a'$ and $C_a'$ depend only on $a$.
Since by \eqref{gen elem} we have
\begin{eqnarray*}
|e^{iv}-e^{iu}|=2|\sin((v-u)/2)|,
\end{eqnarray*}
the required estimates follow now from the last inequality, equality \eqref{gen elem1}, and elementary properties of the function $y=\sin t$.
\end{proof}

\vskip 0.4cm

\section{On some multiplier classes of analytic functions in the unit disk}
\label{S3}

The following notion seems to be quite natural in view of de Leeuw-Rudin theorem and Lemma \ref{l0} and in fact it had already appeared in an explicit form earlier (see e.g. \cite[Lemma~2.1]{dyak}).

\begin{definition}
Let $f\in H^1(\mathbb{D})$. By $L_f$ we denote the set of all measurable on $[0,2\pi]$ functions $h$ with values in $[-1,1]$ such that $h\not\equiv const$ and for some function $g\in H^1(\mathbb{D})$ the following equality holds:
\begin{equation}
\label{pr2muz}
\tilde{g}(t)=\tilde{f}(t)h(t)\;\mbox{for a.e.}\;t\in [0,2\pi),
\end{equation}
where, as above, $\tilde{f}$ and $\tilde{g}$ are the nontangential boundary values of the functions $f$ and $g$ as $z\to e^{it}$, respectively.
\end{definition}

\begin{remark}
\label{rem2a}
If $f\in H(X)$ for some symmetric space $X$ and $h\in L_f$, then the corresponding function $g$ (see \eqref{pr2muz}) also belongs to the space $H(X)$ and $\|g\|_{H(X)}\le \|f\|_{H(X)}$.
\end{remark}

The next result indicates that answer to the question of whether a function $f$ is an extreme point of the ${\rm ball}\,(H(X))$ depends largely on what functions belong to the set $L_f$.

\begin{prop}
\label{motivation of def}
Suppose that the norm of a symmetric space $X$ on $[0,2\pi]$ is strictly monotone, $\|f\|_{H(X)}=1$.

A function $f$ is not an extreme point of the ${\rm ball}\,(H(X))$ if and only if there exists $h\in L_f$ such that for the corresponding function $g\in H^1(\mathbb{D})$ satisfying \eqref{pr2muz}, the inequalities $\|f\pm g\|_{H(X)}\le 1$ hold.


\end{prop}
\begin{proof}
Suppose that $f$ is not an extreme point of ${\rm ball}\,(H(X))$. Then $\|f\pm g\|_{H(X)}\le 1$ for some function $g\in {\rm ball}\,(H(X))$, $g\ne 0$. Since $\|\tilde{f}\|_{X}=\|f\|_{H(X)}=1$, then by the Luzin-Privalov uniqueness theorem (see, for example,
\cite[III.D.$3^0$]{koosis}) $m\{t:\,\tilde{f}(t)=0\}=0$. Moreover, by the condition
$$
\|\tilde{f}\pm \tilde{g}\|_{X}=\|f\pm g\|_{H(X)}\le 1.$$
Therefore, by Lemma \ref{l0}, we have \eqref{pr2muz},
where $h:\,[0,2\pi]\to [-1,1]$. If we assume that $h\equiv c$ for some $c\in [-1,1]$, then we obtain
$$
1\ge\|\tilde{f}\pm \tilde{g}\|_{X}=\|\tilde{f}(1\pm c)\|_{X}=1\pm c,$$
whence $c=0$ and, therefore, $g=0$. Since this contradicts the assumption, $h\not\equiv const$ and hence $h\in L_f$. This completes the proof of this implication, because the function $g\in H^1(\mathbb{D})$ satisfies \eqref{pr2muz} and  inequalities $\|f\pm g\|_{H(X)}\le 1$ hold.

Conversely, let $h\in L_f$ satisfy all the assumptions from the formulation of the proposition. If $g\in H^1(\mathbb{D})$ is defined in \eqref{pr2muz}, then the condition $h\not\equiv const$ guarantees that $g\not\equiv 0$. Thus, $f$ is not an extreme point of the ${\rm ball}\,(H(X))$, and the proposition is proved.

\end{proof}

An inspection of the proof of de Leeuw-Rudin theorem (see \cite[Theorem 1]{deleeuw-rudin} or \cite[Chapter~9]{hoffman}) shows that $L_f=\emptyset$ if and only a given function $f\in H^1$ is outer. Let us turn now to consideration of the problem of identification of functions from the set $L_f$.

\begin{prop}
\label{lemma the most general case}
Let $f\in H^1(\mathbb{D})$ and let 
\begin{equation}
\label{genl3repr}
f(z)=F(z)\cdot I(z),\;\;z\in\D,
\end{equation}
where $F$ is an analytic function and $I$ is a nonconstant inner function.

Then the function
\begin{equation}
\label{genl3}
h(t)=\alpha+\beta\cos({\rm{arg}} (\tilde{I}(t))-\theta),\;\;0\le t\le 2\pi,
\end{equation}
belongs to the set $L_f$ whenever $\alpha, \beta, \theta\in \mathbb{R}$ satisfy the conditions: $\beta\ne 0$ and $|\alpha|+|\beta|\le 1$.
\end{prop}
\begin{proof}
First, it follows from the definition and conditions of the proposition that $h$ is a measurable function on $[0,2\pi]$ that takes values in the interval $[-1,1]$ and $h\not\equiv const$. Furthermore, as is easy to see, for some $\beta',\beta''\in\mathbb{R}$ we have
$$
h(t)=\alpha+2\beta'\cos({\rm{arg}} (\tilde{I}(t)))-2\beta''\sin({\rm{arg}} (\tilde{I}(t))).$$
Since $|\tilde{I}(t)|=1$ for almost all $t\in[0,2\pi)$, a direct calculation shows that
\begin{equation}
\label{useful}
h(t)=\alpha+w\tilde{I}(t)+ \frac{\bar{w}}{\tilde{I}(t)},
\end{equation}
where $w=\beta'+i \beta''$.

Let now $\tilde{g}(t):=h(t)\tilde{f}(t)$, $0\le t\le 2\pi$. 
Due to \eqref{genl3repr}  and \eqref{useful}, we have
\begin{eqnarray*}
\tilde{g}(t)&=& \Big(\alpha+w\tilde{I}(t)+ \frac{\bar{w}}{\tilde{I}(t)}\Big)\cdot \tilde{F}(t)\tilde{I}(t)\\ &=& \big(\alpha\tilde{I}(t)+w\tilde{I}(t)^2+\bar{w}\big)\cdot \tilde{f}(t).
\end{eqnarray*}
Since $\tilde{g}(t)=g(e^{it})$, $0\le t< 2\pi$, where the function
$$
g(z):=\big(\alpha{I}(z)+w{I}(z)^2+\bar{w}\big)\cdot {F}(z),$$
belongs to the space $H^1(\mathbb{D})$, then $h\in L_f$. This proves the proposition.

\end{proof}

Suppose now that the function $f\in H^1(\mathbb{D})$ has zeros. Namely, let $f(a_i)=0$, $i\in J$, where $\{a_i\}_{i\in J}$ is a (finite or infinite) sequence of points in $\mathbb{D}$ satisfying the condition 
\begin{equation}
\label{converg}
\sum_{i\in J} (1-|a_i|)<\infty
\end{equation}
(each factor $(1-|a_i|)$ is repeated in the product as many times as multiplicity of the root $a_i$). Then (see Subsect.\,\ref{Hardy-type spaces}), the corresponding Blaschke product $B_J:=\prod_{i\in J} I_{a_i}$, where for every $a\in \D$ the function $I_a$ is defined by formula  \eqref{factor}, is an inner function such that \eqref{genl3repr}
holds for $I=B_J$ (see e.g. \cite[Chapter IV, $B.2^0$]{koosis}). Thus, by Proposition \ref{lemma the most general case}, for every $\alpha, \beta, \theta\in \mathbb{R}$ such that $\beta\ne 0$ and $|\alpha|+|\beta|\le 1$, the function
\begin{equation*}
\label{genl3}
h(t)=\alpha+\beta\cos({\rm{arg}} (\tilde{B_J}(t))-\theta),\;\;0\le t\le 2\pi,
\end{equation*}
belongs to the set $L_f$.

\begin{remark}
According to equality \eqref{useful}, $h$ is the limit boundary value of the  meromorphic function $H$ in the disk $\mathbb{D}$ defined by
$$
H(z)=\alpha+w{B}_{J}(z)+ \frac{\bar{w}}{{B}_{J}(z)},\;\;z\in\mathbb{D}.$$
\end{remark}

Next, we  single out a special class of functions $f$, for which the set $L_f$ consists only of functions of the form \eqref{genl3repr}. 

Assume that $I$ is a Blaschke factor $I_a$. As above, if $f(a)=0$ for some $a\in\mathbb{D}$, then the same conditions imposed on $\alpha, \beta, \theta\in \mathbb{R}$ ensure that the function
\begin{equation}
\label{partial3}
h(t)=\alpha+\beta\cos({\rm{arg}} (\tilde{I_a}(t))-\theta),\;\;t\in[0,2\pi),
\end{equation}
belongs to the set $L_f$.

\begin{prop}
\label{lemma partial case}
If a function $f\in H^1(\mathbb{D})$ is defined by the formula
\begin{equation}
\label{partial}
f(z)=F(z)I_a(z),
\end{equation}
where $a\in \mathbb{D}$ and $F$ is an outer analytic function, then $h\in L_f$ if and only if $h$ is of the form \eqref{partial3} with $\alpha,\beta,\theta\in\mathbb{R}$ such that $\beta\ne 0$ and $|\alpha|+|\beta|\le 1$.
\end{prop}
\begin{proof}
By the above discussion, we need only to prove that each function $h\in L_f$ can be representable in the form \eqref{partial3}.

Suppose $h\in L_f$. Since $|h(t)|\le 1$, $0\le t<2\pi$, from \eqref{pr2muz} it follows that
$$
|\tilde{g}(t)|\le|\tilde{f}(t)|=|\tilde{F}(t)|,\;\;0\le t<2\pi.$$
Therefore, since $F$ is an outer function, this inequality implies that the function $u(z):=g(z)/F(z)$ belongs to $H^\infty(\mathbb{D})$, $|u(z)|\le 1$. Thus,
\begin{equation}
\label{pr2mu1}
\tilde{g}(t)=\tilde{u}(t)\tilde{F}(t)=\frac{\tilde{u}(t)}{\tilde{I}_a(t)}\cdot \tilde{f}(t),\;\;0\le t<2\pi,
\end{equation}
whence, by \eqref{pr2muz}, it follows that
\begin{equation}
\label{pr2mu1dop}
h(t)=\frac{\tilde{u}(t)}{\tilde{I}_a(t)},\;\;0\le t<2\pi.
\end{equation}

Denote by $I_a^{-1}$ the linear-fractional mapping of the disk $\mathbb{D}$, inverse to $I_a$, and set $v:=u(I_a^{-1})$. Then $v$ is an analytic function in $\mathbb{D}$, $|v(z)|\le 1$, $z\in\mathbb{D}$, and
$$
\frac{{u}(z)}{{I}_a(z)}=\frac{{v}(I_a(z))}{{I}_a(z)}.$$
Hence, in view of \eqref{pr2mu1dop}, it holds
$$
h(t)=\frac{\tilde{u}(t)}{\tilde{I_a}(t)}=\frac{{v}(\tilde{I_a}(t))}{\tilde{I_a}(t)},\;\;0\le t<2\pi.$$
Therefore, from the Taylor expansion $v(z)=\sum_{k=0}^\infty w_kz^k$, it follows 
\begin{equation}
\label{back to}
h(t)=\frac{w_0}{\tilde{I_a}(t)}+w_1+w_2{\tilde{I_a}(t)}+\sum_{k=3}^\infty w_k{\tilde{I_a}(t)}^{k-1}.
\end{equation}

Since the mapping $\tilde{I_a}:\,[0,2\pi)\to \mathbb{T}$ is one-to-one, the inverse mapping $(\tilde{I_a})^{-1}:\,\mathbb{T}\to [0,2\pi)$ is well defined. Moreover, since $h$ is real-valued, so is the function $t\mapsto h((\tilde{I_a})^{-1}(e^{it}))$, $0\le t<2\pi$. Consequently, taking into account that
$$
h((\tilde{I_a})^{-1}(e^{it}))={w_0}e^{-it}+w_1+w_2e^{it}+\sum_{k=3}^\infty w_ke^{i(k-1)t},$$
we have for all $0\le t<2\pi$:
\begin{eqnarray*}
\Im\,(h((\tilde{I_a})^{-1}(e^{it})))&=&(\Im\,(w_0)+\Im\,(w_2))\cos t+(-\Re\,(w_0)+\Re\,(w_2))\sin t+\Im\,(w_1)\\&+&\sum_{k=3}^\infty \big((\Im\,(w_k)\cos((k-1)t)+(\Re\,(w_k)\sin((k-1)t)\big)=0.
\end{eqnarray*}
Hence, $w_1\in\mathbb{R}$, $w_k=0$ for all $k\ge 3$, $\Re (w_2)=\Re (w_0)$ and $\Im (w_2)=-\Im (w_0)$. Thus, because $h$ is real-valued, $\tilde{I}_a(t)=\exp(i\cdot{\rm{arg}} (\tilde{I_a}(t)))$ and equality \eqref{back to} holds, we obtain
$$
h(t)=\Re\,(h(t))=
w_1+2\Re\,(w_0) \cos({\rm{arg}} (\tilde{I_a}(t)))-2\Im\, (w_0)\sin({\rm{arg}} (\tilde{I_a}(t))).$$
Applying now elementary calculations, we arrive at \eqref{partial3} for some $\alpha, \beta,\theta\in\mathbb{R}$. Since the inequality $|\alpha|+|\beta|\le 1$ is a consequence of the relation $|h(t)|\le 1$, $0\le t<2\pi$, the proof is completed.
\end{proof}

\vskip 0.4cm

\section{Characterizations of extreme points of the unit ball of $\hf$}
\label{S4}

From a technical point of view, it will be more convenient to characterize  normalized functions in the space $\hf$ that are not extreme points of the ${\rm ball}\,(\hf)$. In other words, we need to find out conditions under which, for a given function $f\in\hf$, $\|f\|_{\hf}=1$, there exists a non-null function $g\in H^1(\D)$ such that $\|f\pm g\|_{\hf}\le 1$. It is instructive to consider first the partial case when $\mu:=|\tilde{f}|$ is a decreasing function on $[0,2\pi]$.

Let $f\in\hf$, $\|f\|_{\hf}=1$. Assume that there exists a measurable on $[0,2\pi]$ function $h$ with values in $[-1,1]$, $h\not\equiv 0$, satisfying the following conditions:
\begin{equation}
\label{que3}
\int_0^{2\pi}\tilde{f}(t)h(t)e^{int}\,dt=0,\;\;n=1,2,\dots,
\end{equation}
\begin{equation}
\label{que1a}
\mu(t)(1+h(t))\;\;\mbox{and}\;\;\mu(t)(1-h(t))\;\;\mbox{decrease on}\;\;[0,2\pi),
\end{equation}
and
\begin{equation}
\label{que2a}
\int_0^{2\pi}\mu(t)h(t)\,d\varphi(t)=0.
\end{equation}
We show that then $f$ is not an extreme point of the unit ball of $\hf$.

Indeed, condition \eqref{que3} means that the function $\tilde{g}:=h\tilde{f}$ is the nontangential limit of some non-null function $g\in\hf$. Furthermore, since $\mu$ decreases and 
$$
\int_0^{2\pi}\mu(t)\,d\varphi(t)=\|f\|_{\hf}=1,$$
conditions \eqref{que1a} and \eqref{que2a} imply that
\begin{eqnarray*}
\|f\pm g\|_{H(\Lambda(\varphi))} &=&\|\tilde{f}\pm \tilde{g}\|_{\Lambda(\varphi)} =\|\mu(1\pm h)\|_{\Lambda(\varphi)}\\& =&\int_0^{2\pi}\big(\mu(1\pm h)\big)^*(t)\,d\varphi(t)\\& =&\int_0^{2\pi}\mu(t)\,d\varphi(t)\pm \int_{0}^{2\pi}\mu (t)h(t)\,d\varphi(t)\\& =& 1\pm \int_{0}^{2\pi}\mu (t)h(t)\,d\varphi(t)=1.
\end{eqnarray*}
As a result, we get that $f$ is not an extreme point of the ${\rm ball}\,(\hf)$.

Now, we remove the monotonicity assumption imposed on $\mu$ and prove the following criterion.

\begin{theorem}
\label{propos2}
Let $\varphi$ be a strictly increasing and strictly concave function on $[0,2\pi]$. 
Suppose $\|f\|_{\hf}=1$ and $\mu:=|\tilde{f}|$. The following conditions are equivalent:

(a) $f$ is not an extreme point of the unit ball of $\hf$;

(b) there exists a measurable on $[0,2\pi]$ function $h$ with values in $[-1,1]$, $h\not\equiv 0$, satisfying condition \eqref{que3} and such that
\begin{equation}
\label{que1}
E_t(\mu)=E_t(\mu(1+h))=E_t(\mu(1-h))\;\;\mbox{for all}\;\;t\in[0,2\pi],
\end{equation}
and
\begin{equation}
\label{que2}
\int_0^{2\pi}\Big(\int_{E_t(\mu)}\mu(s)h(s)\,ds\Big)_t'\,d\varphi(t)=0.
\end{equation}
\end{theorem}


\begin{proof}
$(a)\Rightarrow (b)$.
First, by Proposition \ref{motivation of def}, there exists $h\in L_f$ such that for the  function $g\in H^1(\mathbb{D})$ satisfying \eqref{pr2muz}, we have
\begin{equation}
\label{one more eq}
\|f\pm g\|_{\hf}=1.
\end{equation}
Therefore, since $|\widetilde{(f\pm g)}(s)|=\mu(s)(1\pm h(s))$, $0<s\le 2\pi$, it holds
$$
\|\mu\pm \mu h)\|_{\Lambda(\varphi)}=\|{f}(1\pm h)\|_{\hf}=1,$$
and, by Proposition \ref{propos}, we obtain \eqref{que1}. Hence,
\begin{equation}
\label{pr1mu}
\int_{E_t}\mu(s)\,ds=\int_0^t \mu^*(s)\,ds,\;\;0<t\le 2\pi,
\end{equation}
and
\begin{equation}
\label{pr2mu}
\int_{E_t}\mu(s)(1\pm h(s))\,ds=\int_0^t (\mu(1\pm h))^*(s)\,ds,\;\;0<t\le 2\pi,
\end{equation}
where $E_t:=E_t(\mu)$, $0<t\le 2\pi$.

Let
$$
F_+(t):=\int_0^t (\mu(1+h))^*(s) ds,\;F_-(t):=\int_0^t (\mu(1-h))^*(s) ds, \quad t\in [0,2 \pi].
$$
Then, integrating by parts and applying \eqref{pr2mu} and \eqref{pr1mu},  we obtain
\begin{eqnarray*}
\|f+ g\|_{H(\Lambda(\varphi))} &=& \|\mu(1+ h)\|_{\Lambda(\varphi)}=\int_0^{2\pi}(\mu(1+ h))^*(t)\,d\varphi(t)\\&=&\int_0^{2\pi} F_+(t)\,d(- \varphi'(t))
+ \varphi'(2\pi) F_+(2\pi)\\&=&
\int_0^{2\pi} \int_{E_t}\mu(s)(1+h(s))\,ds\,d(- \varphi'(t))\\&+&
\varphi'(2\pi) \int_0^{2\pi}\mu(t)(1+h(t))\,dt\\&=&
\int_0^{2\pi} \int_{0}^t\mu^*(s)\,ds\,d(- \varphi'(t))+ \varphi'(2\pi) \int_0^{2\pi}\mu^*(t)\,dt\\&+&\int_0^{2\pi} \int_{E_t}\mu(s)h(s)\,ds\,d(- \varphi'(t))+\varphi'(2\pi) \int_0^{2\pi}\mu(t)h(t)\,dt.
\end{eqnarray*}
Next, because $E_{2\pi}=[0,2\pi]$ and $\|f\|_{H(\Lambda(\varphi))}=1$, by going in the opposite direction, we arrive at the equality:
\begin{eqnarray}
\|f+ g\|_{H(\Lambda(\varphi))} &=&\|f\|_{H(\Lambda(\varphi))}+\int_0^{2\pi}\Big(\int_{E_t}\mu(s)h(s)\,ds\Big)_t'\,d\varphi(t)\nonumber\\&=& 1+\int_0^{2\pi}\Big(\int_{E_t}\mu(s)h(s)\,ds\Big)_t'\,d\varphi(t).
\label{spec1}
\end{eqnarray}
In the same way, it can be shown also that
\begin{equation}
\label{spec1a}
\|f- g\|_{H(\Lambda(\varphi))} =1-\int_0^{2\pi}\Big(\int_{E_t}\mu(s)h(s)\,ds\Big)_t'\,d\varphi(t).
\end{equation}
As a result, from \eqref{spec1} and \eqref{spec1a} it follows the equivalence of  relations \eqref{que2} and \eqref{one more eq}.

Finally, since $g$ is an analytic function in the disk, by \eqref{pr2muz}, we arrive at \eqref{que3}:
$$
\int_0^{2\pi}\tilde{f}(t)h(t)e^{int}\,dt=\int_0^{2\pi}\tilde{g}(t)e^{int}\,dt=0,\;\;n=1,2,\dots.
$$

$(b)\Rightarrow (a)$.
By \eqref{que3} (see also the discussion before the theorem) the function $\tilde{g}:=h\tilde{f}$ is the nontangential limit of some non-null function $g\in\hf$. Further, condition \eqref{que1} implies relations \eqref{pr1mu} and \eqref{pr2mu}. Consequently, precisely as in the first part of the proof, we obtain equalities \eqref{spec1} and \eqref{spec1a}, which imply that relations \eqref{one more eq} and \eqref{que2} are equivalent. Summing up, we conclude that $f$ is not an extreme point of the unit ball of $\hf$.

\end{proof}

Coming back to the partial case, considered before Theorem \ref{propos2}, we obtain the following.

\begin{corollary}
\label{corollary from propos2}
Assume that all the conditions of Theorem \ref{propos2} are satisfied, and, additionally, let the function $\mu$ decrease on the interval $[0,2\pi)$.
Then, $f$ is not an extreme point of the ${\rm ball}\,(\hf)$ if and only if there exists a measurable function $h:\,[0,2\pi]\to [-1,1]$ satisfying conditions \eqref{que3}, \eqref{que1a}, and \eqref{que2a}.
\end{corollary}
\begin{proof}
In view of Theorem \ref{propos2}, it suffices to check that conditions \eqref{que1} and \eqref{que2} are equivalent to conditions \eqref{que1a} and \eqref{que2a}, respectively.

First, let \eqref{que1} and \eqref{que2} be satisfied. 
Since $\mu$ decreases, we can assume that $\mu^*=\mu$ and hence $E_t=(0,t)$, $0\le t\le 2\pi$. Therefore, by \eqref{que1}, we have
\begin{equation*}
\int_{0}^t\mu(s)(1\pm h(s))\,ds=\int_0^t (\mu(1\pm h))^*(s)\,ds,\;\;0<t\le 2\pi,
\end{equation*}
whence
$$
(\mu(1\pm h))^*=\mu(1\pm h)\;\;\mbox{a.e. on}\;[0.2\pi).$$
Thus, \eqref{que1a} is proved. Also, since 
\begin{equation}
\label{triv}
\mu(t)h(t)=\Big(\int_{0}^t \mu(s)h(s)\,ds\Big)_t'=
\Big(\int_{E_t}\mu(s)h(s)\,ds\Big)_t'\;\;\mbox{a.e. on}\;[0.2\pi),
\end{equation}
\eqref{que2} implies \eqref{que2a}.

Conversely, because $\mu$ decreases, from \eqref{que1a} it follows that
$$
E_t(\mu)=E_t(\mu(1+h))=E_t(\mu(1-h))=(0,t),\;\;0\le t\le 2\pi,$$
which implies \eqref{que1}. Moreover, observe that equality \eqref{triv} still holds. Combining this together with \eqref{que2a}, we obtain \eqref{que2}.

\end{proof}

Given a function $f\in H(\Lambda(\varphi))$ we denote by $\mathcal{I}(f)$ the set of all nonconstant inner functions $I$, for which equality \eqref{genl3repr} holds for some analytic function $F$.

By using Propositions \ref{lemma the most general case} and \ref{lemma partial case}
from the previous section, the last results can be specified as follows.

\begin{theorem}\label{Th3a}
Let $\varphi$ be a strictly increasing and strictly concave function on $[0,2\pi]$. 
Assume also that $f\in H(\Lambda(\varphi))$, $\|f\|_{H(\Lambda(\varphi))}=1$. 

Consider two conditions:

(i) $f$ is not an extreme point of the unit ball in $H(\Lambda(\varphi))$;

(ii) there exist $I\in\mathcal{I}(f)$ and $\alpha, \beta, \theta\in\mathbb{R}$ such that $\beta\ne 0$, $|\alpha|+|\beta|\le 1$ and the function
$$
h(t)=\alpha+\beta\cos({\rm{arg}} (\tilde{I}(t))-\theta),\;\;t\in [0,2\pi),$$ satisfies \eqref{que1} as well as the following condition:
\begin{equation}
\label{cond b}
\int_0^{2\pi}\Big(\int_{E_t(\mu)}\mu(s)\cos({\rm{arg}} (\tilde{I}(s))-\theta)\,ds\Big)_t'\,d\varphi(t)=-\frac{\alpha}{\beta},
\end{equation}
where, as above, $\mu:=|\tilde{f}|$.

Then, the implication $(ii)\Rightarrow (i)$ holds. 

Moreover, if $f$ is of the form \eqref{partial}, where $F$ is an outer function, the converse holds as well, i.e.,  condition $(i)$ implies that there exist $\alpha, \beta, \theta\in\mathbb{R}$ such that $\beta\ne 0$, $|\alpha|+|\beta|\le 1$, and the function $h$ defined by \eqref{partial3} 
(with $I_a$ from \eqref{partial}) satisfies all the conditions from (ii).  Also, the inequalities  $\|f\pm g\|_{\hf}=1$, with $g\in \hf$ such that $\tilde{g}=h\tilde{f}$,  hold.

\end{theorem}
\begin{proof}
First, we prove that conditions \eqref{cond b} and \eqref{que2} are equivalent.
Indeed, by definition of the set $E_t(\mu)$, it follows
\begin{eqnarray*}
\Big(\int_{E_t(\mu)}\mu(s)h(s)\,ds\Big)_t'&=&\Big(\alpha\int_{0}^t\mu^*(s)\,ds+\beta\int_{E_t(\mu)}\mu(s)\cos({\rm{arg}} (\tilde{I}(s))-\theta)\,ds\Big)_t'\\&=& \alpha\mu^*(t)+\beta\Big(\int_{E_t(\mu)}\mu(s)\cos({\rm{arg}} (\tilde{I}(s))-\theta)\,ds\Big)_t'.
\end{eqnarray*}
Hence, the equality $\|\mu\|_{\Lambda(\varphi)}=\|f\|_{H(\Lambda(\varphi))}=1$ implies that
\begin{eqnarray*}
\int_0^{2\pi}\Big(\int_{E_t(\mu)}\mu(s)h(s)\,ds\Big)_t'\,d\varphi(t)&=& \alpha\int_0^{2\pi}\mu^*(t)\,d\varphi(t)\\&+&\beta\int_0^{2\pi}\Big(\int_{E_t(\mu)}\mu(s)\cos({\rm{arg}} (\tilde{I}(s)-\theta)\,ds\Big)_t'\,d\varphi(t)\\&=& \alpha+\beta\int_0^{2\pi}\Big(\int_{E_t(\mu)}\mu(s)\cos({\rm{arg}} (\tilde{I}(s))-\theta)\,ds\Big)_t'\,d\varphi(t),
\end{eqnarray*}
and the required equivalence follows.


We proceed now with proving the implication $(ii)\Rightarrow (i)$.
By Proposition \ref{lemma the most general case}, the function $h$ from the formulation of the theorem belongs to the set $L_f$. Therefore, since $\tilde{f}h=\tilde{g}$ for some function $g\in H^1(\mathbb{D})$, we get \eqref{que3}.

Next, since \eqref{cond b} and \eqref{que2} are equivalent, all the  conditions \eqref{que3}, \eqref{que1}, and \eqref{que2} are satisfied. Therefore, according to Theorem \ref{propos2}, $f$ is not an extreme point of the unit ball in  $H(\Lambda(\varphi))$. Thus, $(i)$ is proved.

Let us prove the second assertion of the theorem. Let $(i)$ be satisfied. 
By Theorem \ref{propos2} there exists a measurable function $h$ with values in $[-1,1]$ satisfying conditions \eqref{que3}, \eqref{que1}, and \eqref{que2}. In particular, it follows from \eqref{que3} that $h\in L_f$. Consequently, by Proposition \ref{lemma partial case}, $h$ is a function of the form \eqref{partial3} for some $\alpha, \beta,\theta\in\mathbb{R}$, $\beta\ne 0$ and $|\alpha|+|\beta|\le 1$. As was already noted in the first part of the proof, \eqref{cond b} is a consequence of \eqref{que2}. This completes the proof.
\end{proof}

From Theorem \ref{Th3a} and Corollary \ref{corollary from propos2} it follows

\begin{corollary}\label{corTh3a}
Suppose that all the conditions of Theorem \ref{Th3a} are satisfied, and, additionally, the function $\mu$ decreases on the interval $[0,2\pi)$.

Consider two conditions:

(i) $f$ is not an extreme point of the unit ball in $H(\Lambda(\varphi))$;

(ii) there exist $\alpha, \beta, \theta\in\mathbb{R}$ such that $\beta\ne 0$, $|\alpha|+|\beta|\le 1$, both functions $\mu(t)\cdot(1\pm (\alpha+\beta\cos({\rm{arg}} (\tilde{I}(t))-\theta)))$ decrease on $[0,2\pi)$ and
$$
\int_0^{2\pi}\mu(t)\cos({\rm{arg}} (\tilde{I}(t))-\theta)\,d\varphi(t)=-\frac{\alpha}{\beta}.$$

Then, the implication $(ii)\Rightarrow (i)$ holds. Moreover, if $f$ has the form \eqref{partial}, where $F$ is an outer function,   condition $(i)$ implies that there exist $\alpha, \beta, \theta\in\mathbb{R}$ such that $\beta\ne 0$, $|\alpha|+|\beta|\le 1$, and the function $h$ defined by \eqref{partial3} (with $I$ from \eqref{partial}) satisfies all the conditions from (ii).
\end{corollary}

\vskip 0.4cm

\section{Main results}
\label{MAIN}

Given an inner function $I$ in the disk $\D$ and a measurable mapping $\omega:\,[0,2\pi]\to [0,2\pi]$ we introduce the following notation:
$$\xi_{I}(t):={\rm{arg}}(\tilde{I}(t)),\;\;t\in \mathbb{R},$$
$$\xi_{I,\omega}(t):=\xi_I({\omega}(t))={\rm{arg}}(\tilde{I}({\omega}(t))),\;\;t\in \mathbb{R},$$
\begin{equation}
\label{quant1}
\xi_{I,\omega}^1(u,v):=|\sin((\xi_{I,\omega}(u)-\xi_{I,\omega}(v))/2)|,\;\;u,v\in \mathbb{R},
\end{equation}
and
\begin{equation}
\label{quant2}
\xi_{I,\omega}^2(u,v,\gamma):=|\sin((\xi_{I,\omega}(u)+\xi_{I,\omega}(v))/2-\gamma)|,\;\;u,v,\gamma\in \mathbb{R}.
\end{equation}

We start the section with the following key technical result.

\begin{prop}\label{Th4a new}
Let $\nu$ be a decreasing, nonnegative function on $[0,2\pi]$ (if $\lim_{t\to +0}\nu(t)=\infty$ we put $\nu(0)=\infty$) and let $I$ be an nonconstant inner function in the disk $\D$. 
For any $a\in\mathbb{D}$ and measurable mapping $\omega:\,[0,2\pi]\to [0,2\pi]$ the following conditions are equivalent:

(a) there exists $\gamma\in \mathbb{R}$ such that for every $t\in [0,2\pi]$
\begin{equation} \label{main weight}
\limsup_{u\to t,v\to t,0<u<v\le 2\pi}\frac{\nu(v)\xi_{I,\omega}^1(u,v)\xi_{I,\omega}^2(u,v,\gamma)}{\nu(u)-\nu(v)}<\infty;
\end{equation}

(b) there exist $\alpha, \beta, \theta\in\mathbb{R}$ such that $\beta\ne 0$, $|\alpha|+|\beta|\le 1$, and both functions $\eta_{I,\omega}^+$ and $\eta_{I,\omega}^-$ defined by the equality:
$$
\eta_{I,\omega}^\pm(u):=\nu(u)\cdot(1\pm( \alpha+ \beta\cos(\xi_{I,\omega}(u)-\theta)))$$
decrease on the interval $(0,2\pi]$.

Furthermore, if  condition (a) is satisfied and $\varphi$ is a concave, increasing function on $[0,2\pi]$ such that $\int_0^{2\pi}\nu(t)\,d\varphi(t)=1$, then $\alpha, \beta, \theta\in\mathbb{R}$ can be chosen in such a way that along with monotonicity of the functions $\eta_{I,\omega}^\pm$ the following holds:
\begin{equation} \label{main weight_b}
\int_0^{2\pi}\nu(u)\cos(\xi_{I,\omega}(u)-\theta)\,d\varphi(u)=-\frac{\alpha}{\beta}.
\end{equation}

\end{prop}

\begin{remark}
Condition \eqref{main weight} is assumed to be not satisfied if $\nu(u)=\nu(v)$ for some $u,v\in (0,2\pi]$, $u<v$, or equivalently
$m\{s\in (0,2\pi]:\,\nu(s)=c\}>0$ for some $c$ (see for this case \cite{A-prep}). 

\end{remark}

\begin{proof}[Proof of Proposition \ref{Th4a new}]

For brevity, we put
$$
\xi:=\xi_{I,\omega}, \xi^1:=\xi_{I,\omega}^1, \xi^2:=\xi_{I,\omega}^2, \eta^\pm:=\eta_{I,\omega}^\pm.$$

First, observe that for any $\alpha, \beta, \theta\in\mathbb{R}$ and $0<u<v\le 2\pi$ 
\begin{eqnarray}
\eta^\pm(u)-\eta^\pm(v) &=& (\nu(u)-\nu(v))(1\pm( \alpha+ \beta\cos(\xi(u)-\theta))\nonumber\\&\pm& \beta\nu(v)(\cos(\xi(u)-\theta)-\cos(\xi(v)-\theta))\nonumber\\&=&
(\nu(u)-\nu(v))(1\pm( \alpha+ \beta\cos(\xi(u)-\theta)))\nonumber\\&\pm& 2 \beta\nu(v)\sin((\xi(v)-\xi(u))/2)\nonumber\\
&\times&\sin((\xi(v)+\xi(u))/2-\theta).
\label{start1}
\end{eqnarray}

$(a)\Rightarrow (b)$.
For definiteness, we will assume that $\nu(0)=\infty$ (if $\nu(0)<\infty$, the proof is somewhat simplified).

By the assumption, there exists $\gamma>0$ such that for each $t\in [0,2\pi]$ we can find $C_t>0$ and $\delta(t)>0$ such that the inequality
\begin{equation*}
\label{extra rel}
\nu(v)\xi^1(u,v)\xi^2(u,v,\gamma)\le C_t(\nu(u)-\nu(v))
\end{equation*}
holds for all $0<u<v\le 2\pi$ such that $u,v\in (t-\delta(t),t+\delta(t))$.  
Therefore, setting $\theta$ to be equal to $\gamma$
and taking into account the definition of $\xi^1(u,v)$ and $\xi^2(u,v,\gamma)$,
for all such $u$ and $v$ we obtain
\begin{equation*}
\nu(v)|\sin((\xi(u)-\xi(v))/2)|\cdot|\sin((\xi(v)+\xi(u))/2)-\theta)|
\le C_t(\nu(u)-\nu(v)).
\end{equation*}
In turn, combining this together with equality \eqref{start1}, for all $u,v\in (t-\delta(t),t+\delta(t))$, $0<u<v\le 2\pi$, we infer that
\begin{eqnarray*}
\eta^\pm(u)-\eta^\pm(v) &\ge& (\nu(u)-\nu(v))(1-|\alpha|- |\beta|)\\&-&2 |\beta|\nu(v)|\sin((\xi(v)-\xi(u))/2)|\\ &\times&|\sin((\xi(v)+\xi(u))/2-\theta)|
\\&\ge&
(\nu(u)-\nu(v))(1-|\alpha|- (1+2C_t)|\beta|).
\end{eqnarray*}
If the absolute values of the numbers $\alpha=\alpha(t)$ and $\beta=\beta(t)$ are sufficiently small, we have $|\alpha|+ (1+2C_t)|\beta|<1$. Then, since $\nu(u)\ge \nu(v)$, the last inequality implies that both functions $\eta^+$ and $\eta^-$ are decreasing on the set $(t-\delta(t),t+\delta(t))\cap (0,2\pi]$. In particular, this is valid for some interval $(0,\delta(0))$ with $\delta(0)>0$ whenever $\alpha(0)$ and $\beta(0)$ are sufficiently small.

Next, by the compactness, we select a finite subcovering from the open covering $\{(t-\delta(t),t+\delta(t))\}_{t\in [\delta(0),2\pi]}$ of the interval $[\delta(0),2\pi]$. Hence, the interval $[\delta(0),2\pi]$ is contained into a finite union of overlapping intervals, on each of those both functions $\eta^+$ and $\eta^-$ decrease if the corresponding values of $|\alpha|$ and $|\beta|$ are sufficiently small. Choosing $\alpha$ and $\beta$ with smaller absolute values than all of them as well as $\alpha(0)$ and $\beta(0)$ have, we get that the functions $\eta^+$ and $\eta^-$ decrease, for these values of $\alpha$ and $\beta$, on the entire interval $(0,2\pi]$.

In addition, since
$$
\Big|\int_0^{2\pi}\nu(t)\cos(\xi(u)-\theta)\,d\varphi(t)\Big|\le \int_0^{2\pi}\nu(t)\,d\varphi(t)=1,$$
changing, if necessary, the values of $\alpha$, $\beta$ (without increasing their absolute values), we obtain that
$$
\int_0^{2\pi}\nu(t)\cos(\xi(u)-\theta)\,d\varphi(t)=-\frac{\alpha}{\beta},$$
and equality \eqref{main weight_b} also follows.

$(b)\Rightarrow (a)$. Assume that both functions $\eta^+$ and $\eta^-$ decrease on the interval $(0,2\pi]$ for some $\alpha, \beta, \theta\in\mathbb{R}$, $\beta\ne 0$, $|\alpha|+|\beta|\le 1$, but condition \eqref{main weight} is not satisfied. Then there exist $t\in [0,2\pi]$ and sequences $\{u_n\}_{n=1}^\infty$, $\{v_n\}_{n=1}^\infty$ such that $0<u_n<v_n\le 2\pi$, $n=1,2,\dots$, $\lim_{n\to\infty}u_n=\lim_{n\to\infty}v_n=t$ and for all $n\in\mathbb{N}$
$$
\nu(v_n)\xi^1(u_n,v_n)\xi^2(u_n,v_n,\theta)\ge n(\nu(u_n)-\nu(v_n)),$$
or equivalently
\begin{equation} \label{main on the first}
\nu(v_n)|\sin((\xi(u_n)-\xi(v_n))/2)||\sin((\xi(u_n)+\xi(v_n))/2-\theta)|\ge n(\nu(u_n)-\nu(v_n)).
\end{equation}

On the other hand, since $|\alpha|+|\beta|\le 1$, by the monotonicity of $\eta^+$ and $\eta^-$ and by equality \eqref{start1}, for all $n\in\mathbb{N}$ and both signs we have:
\begin{eqnarray*}
\pm 2 \beta\nu(v_n)\sin((\xi(u_n)&-&\xi(v_n))/2)\sin((\xi(v_n)+\xi(u_n))/2-\theta)\le\\&\le&(\nu(u)-\nu(v))(1\pm( \alpha+ \beta\cos(\xi(u)-\theta)))
\\&\le&(\nu(u_n)-\nu(v_n))(1+|\alpha|+ |\beta|)\\&\le& 2(\nu(u_n)-\nu(v_n)).
\end{eqnarray*}
Choosing the sign on the left-hand side of this inequality in appropriate way for each $n=1,2,\dots$, we obtain that
$$
\nu(v_n)|\sin((\xi(u_n)-\xi(v_n))/2)||\sin((\xi(u_n)+\xi(v_n))/2-\theta)|\le \frac{1}{|\beta|}(\nu(u_n)-\nu(v_n)),\;\;n=1,2,\dots.
$$
Since the last inequality contradicts \eqref{main on the first}, the implication $(b)\Rightarrow (a)$ and hence the proposition are proved.

\end{proof}



\begin{theorem}\label{cor Th4a new}
Let $\varphi$ be a strictly increasing and strictly concave function on $[0,2\pi]$.
Assume that $f$ is an extreme point of the ${\rm ball}\,(H(\Lambda(\varphi)))$ such that there exists a one-to-one measure-preserving mapping $\omega:\,[0,2\pi]\to [0,2\pi]$ such that $\mu^*=\mu(\omega)$ (as above, $\mu:=|\tilde{f}|$). 

Then, for every function $I\in\mathcal{I}(f)$ and each $\gamma\in \mathbb{R}$ there exists a point $t\in [0,2\pi]$ such that
\begin{equation} \label{main weight 2}
\liminf_{u\to t,v\to t,0<u<v\le 2\pi}\frac{\mu^*(u)-\mu^*(v)}{\mu^*(v)\xi_{I,\omega}^1(u,v)\xi_{I,\omega}^2(u,v,\gamma)}=0.
\end{equation}

Moreover, if $f$ is a function of norm one in $\hf$ defined by \eqref{partial}, where $F$ is an outer analytic function and $I_a$, $a\in\D$, is a Blaschke factor, then $f$ is an extreme point of the ${\rm ball}\,(H(\Lambda(\varphi)))$ if and only if for $I=I_a$ and for each $\gamma\in \mathbb{R}$ there exists a point $t\in [0,2\pi]$ such that  \eqref{main weight 2} holds.
\end{theorem}
\begin{proof}
We start with proving the first assertion. Let $f$ be an extreme point of the set ${\rm ball}\,(H(\Lambda(\varphi)))$, satisfying the assumptions of the theorem. To the contrary, assume that for some function $I\in\mathcal{I}(f)$ there exists $\gamma\in \mathbb{R}$ such that  relation \eqref{main weight 2} does not hold for all $t\in [0,2\pi]$. Clearly, setting $\nu=\mu^*$, we see that relation \eqref{main weight} holds for the same $I$ and $\gamma\in \mathbb{R}$ and for any $t\in [0,2\pi]$. 
Consequently, according to Proposition \ref{Th4a new}, there exist $\alpha, \beta, \theta\in\mathbb{R}$ such that $\beta\ne 0$, $|\alpha|+|\beta|\le 1$, and both functions
$$
\eta_{I,\omega}^\pm(u):=\mu^*(u)\cdot(1\pm( \alpha+ \beta\cos(\xi_{I,\omega}(u)-\theta)))$$
decrease on the interval $(0,2\pi]$ and 
\begin{equation} 
\label{main weight_b_new}
\int_0^{2\pi}\mu^*(u)\cos(\xi_{I,\omega}(u)-\theta)\,d\varphi(u)=-\frac{\alpha}{\beta}.
\end{equation}
Hence, the functions $g^*\cdot (1+h_\omega)$ and $g^*\cdot (1-h_\omega)$, where 
$$
g(s):=\mu(s),\; h(s):= \alpha+ \beta\cos(\xi_I(s)-\theta)\;\;\mbox{and}\;\;h_\omega(s):=h({\omega}(s)),$$ 
satisfy condition (iii) of Proposition \ref{reduction4} with $T=2\pi$.
Applying this proposition, we obtain that
\begin{equation*}
E_t(\mu\cdot(1\pm(\alpha+ \beta\cos(\xi_I-\theta))))=E_t(\mu),\;\;0\le t\le 2\pi,
\end{equation*}
and 
\begin{equation}
\label{main weight_b_newnew}
\Big(\int_{E_t(\mu)}\mu(s)(\alpha+ \beta\cos(\xi_I(s)-\theta))\,ds\Big)_t'=\mu^*(t)(\alpha+ \beta\cos(\xi_{I,\omega}(t)-\theta))
\end{equation}
a.e. on $[0.2\pi]$. 

On the one hand, since $\|\mu\|_{\Lambda(\varphi)}=\|f\|_{H(\Lambda(\varphi))}=1$, by definition of the set $E_t(\mu)$, we have
\begin{eqnarray*}
\int_0^{2\pi}\Big(\int_{E_t(\mu)}\mu(s)(\alpha &+& \beta\cos(\xi_I(s)-\theta))\,ds\Big)_t'\,d\varphi(t)=\alpha\int_0^{2\pi}\Big(\int_{0}^t\mu^*(s)\,ds\Big)_t'\,d\varphi(t)\\&+&\beta\int_0^{2\pi}\Big(\int_{E_t(\mu)}\mu(s)\cos(\xi_I(s)-\theta)\,ds\Big)_t'\,d\varphi(t)\\&=& \alpha+\beta\int_0^{2\pi}\Big(\int_{E_t(\mu)}\mu(s)\cos(\xi_I(s)-\theta)\,ds\Big)_t'\,d\varphi(t).
\end{eqnarray*}
On the other hand, in view of \eqref{main weight_b_new} and \eqref{main weight_b_newnew}, 
\begin{eqnarray*}
\int_0^{2\pi}\Big(\int_{E_t(\mu)}\mu(s)(\alpha &+& \beta\cos(\xi_I(s)-\theta))\,ds\Big)_t'\,d\varphi(t)=\alpha\int_0^{2\pi}\mu^*(t)\,d\varphi(t)\\&+&\beta\int_0^{2\pi}\mu^*(t)\cos(\xi_{I,\omega}(t)-\theta)\,d\varphi(t)=\alpha-\alpha= 0.
\end{eqnarray*}
From two last equalities and the fact that $\beta\ne 0$ it follows 
$$
\int_0^{2\pi}\Big(\int_{E_t(\mu)}\mu(s)\cos(\xi_I(s)-\theta)\,ds\Big)_t'\,d\varphi(t)=-\frac{\alpha}{\beta}.$$

As a result, all the conditions of part (ii) of Theorem \ref{Th3a} are satisfied, and hence $f$ is not an extreme point of the unit ball in the space $H(\Lambda(\varphi))$. Since this contradicts the assumption, the first assertion of the theorem is proved.

Let us proceed with the proof of the second assertion. Clearly, it suffices to show that every normalized in $H(\Lambda(\varphi))$ function $f$ defined by \eqref{partial} is an extreme point of the unit ball in the space $H(\Lambda(\varphi))$ whenever for $I=I_a$ and each $\gamma\in \mathbb{R}$ there exists a point $t\in [0,2\pi]$ such that  \eqref{main weight 2} holds.

Assume that is not the case, that is, there is a function $f$ that satisfies the above conditions but that is not an extreme point of the ${\rm ball}\,({\hf})$. Then, arguing similarly as in the first part of the proof and applying successively the second part of Theorem \ref{Th3a}, Propositions \ref{reduction4} and \ref{Th4a new}, we obtain a contradiction.
Indeed, first there exist $\alpha, \beta, \theta\in\mathbb{R}$ such that $\beta\ne 0$, $|\alpha|+|\beta|\le 1$, and the function $h$ defined by \eqref{partial3} satisfies all the conditions from part (ii) of Theorem \ref{Th3a}. Next, applying Proposition \ref{reduction4}, we get that both functions 
$\eta_\omega^{\pm}=\mu^*(1\pm h_\omega)$, where $\mu:=|\tilde{f}|$,  decrease on $[0,2\pi]$. Hence, by Proposition \ref{Th4a new}, if $\nu=\mu^*$ and $I=I_a$, relation \eqref{main weight} holds for some $\gamma\in \mathbb{R}$ and all $t\in [0,2\pi]$, or equivalently, relation \eqref{main weight 2} fails for all $t\in [0,2\pi]$. Since this contradicts the assumption, we done.
\end{proof}

\vskip0.4cm

For any non-negative function $\mu:\,[0,2\pi]\to\mathbb{R}$ and a mapping $\omega:\,[0,2\pi]\to[0,2\pi]$, we define two subsets of the interval $[0,2\pi]$. By $\mathcal{E}_1=\mathcal{E}_1(\mu,\omega)$ and $\mathcal{E}_2=\mathcal{E}_2(\mu,\omega)$ we denote the set of all $t\in [0,2\pi]$ such that
\begin{equation} \label{ineq5b2}
\liminf_{u\to t,v\to t,0<u<v\le 2\pi}\frac{\mu^*(u)-\mu^*(v)}{\mu^*(v)(|{\omega}(v)-{\omega}(u)|)}=0,
\end{equation}
and 
\begin{equation} \label{ineq5b_gen}
\liminf_{u\to t,v\to t,0<u<v\le 2\pi}\frac{\mu^*(u)-\mu^*(v)}{\mu^*(v)({\omega}(v)-{\omega}(u))\max(|{\omega}(u)-{\omega}(t)|,|{\omega}(v)-{\omega}(t)|)}=0,
\end{equation}
respectively. As is easy to see, $\mathcal{E}_2 \subset \mathcal{E}_1$.

Furthermore, as above, for every inner function $I$ we set $\xi_{I,\omega}(t):={\rm{arg}}(\tilde{I}({\omega}(t))$, $t\in \mathbb{R}$.
Then, given a point $t\in [0,2\pi]$ by $R_{I,\omega}(t)$ we denote the set of all $s\in [0,2\pi]$ such that 
the difference $\xi_{I,\omega}(s)-\xi_{I,\omega}(t)$ is a multiple of $\pi$ (equivalently, $\tilde{I}({\omega}(s))=\tilde{I}({\omega}(t))$ or $\tilde{I}({\omega}(s))=-\tilde{I}({\omega}(t))$).
It is clear that $t\in R_{I,\omega}(t)$ for every $t\in [0,2\pi]$ and also either $R_{I,\omega}(t_1)=R_{I,\omega}(t_2)$ or $R_{I,\omega}(t_1)\cap R_{I,\omega}(t_2)=\emptyset$.

\begin{theorem}\label{Th4a general}
Let $\varphi$ be a strictly increasing and strictly concave function on $[0,2\pi]$.
Suppose that $\|f\|_{H(\Lambda(\varphi))}=1$ and there exists a measure-preserving one-to-one mapping $\omega:\,[0,2\pi]\to [0,2\pi]$ such that $\mu^*=\mu(\omega)$, where $\mu:=|\tilde{f}|$.

Then, if $f$ is an extreme point of the unit ball in the space $H(\Lambda(\varphi))$, at least one of the following two conditions is satisfied:

(a) there are no function $I\in\mathcal{I}(f)$ and $t_0\in [0,2\pi]$ such that $\mathcal{E}_1(\mu,\omega)\subset R_{I,\omega}(t_0)$;

(b) $\mathcal{E}_2(\mu,\omega)\ne\emptyset$.

\end{theorem}

 \begin{proof}
If $f$ is an outer function, we need to prove nothing, because of then $\mathcal{I}(f)=\emptyset$ and so the condition (a) holds. So, we may assume that there is a nonconstant inner function $I$ such that $f=F\cdot I$, where $F\in H^1$.  

By the assumption and Theorem \ref{cor Th4a new}, for every function $I\in\mathcal{I}(f)$ and each $\gamma\in \mathbb{R}$ there exists a point $t\in [0,2\pi]$ such that \eqref{main weight 2} holds,
where
$$
\xi_{I,\omega}^1(u,v)=|\sin(({\rm{arg}} \tilde{I}({\omega}(v))-{\rm{arg}} \tilde{I}({\omega}(u)))/2)|,$$
$$
\xi_{I,\omega}^2(u,v,\gamma)=|(\sin(({\rm{arg}} \tilde{I}({\omega}(v))+{\rm{arg}} \tilde{I}({\omega}(u)))/2-\gamma)|
$$
(see \eqref{quant1} and \eqref{quant2}). Thus, the theorem will be proved if we show that \eqref{main weight 2} implies at least one of conditions (a) and (b).

Denote by $\Phi$ the set of all $t\in [0,2\pi]$ such that \eqref{main weight 2} holds for some $\gamma\in \mathbb{R}$. Then, we have the following:

(i) $\Phi$ contains at least two points $t_1,t_2\in [0,2\pi]$ such that $R_{I,\omega}(t_1)\ne R_{I,\omega}(t_2)$

or

(ii) $\Phi\subset R_{I,\omega}(t_0)$ for some point $t_0\in [0,2\pi]$.

Assume first that condition (i) is satisfied. Then, since $0\le \xi_{I,\omega}^2(u,v,\gamma)\le 1$, by \eqref{main weight 2}, we obtain for the points $t_1$ and $t_2$ that
\begin{equation*}
\liminf_{u\to t_i,v\to t_i,0<u<v\le 2\pi}\frac{\mu^*(u)-\mu^*(v)}{\mu^*(v)\xi_{I,\omega}^1(u,v)}=0,\;\;i=1,2.
\end{equation*}
This relation combined with the inequality 
\begin{equation} \label{main weight 2a}
\xi_{I,\omega}^1(u,v)\le C_a|{\omega}(u)-{\omega}(v)|,
\end{equation}
(see Lemma \ref{lemma: calcul}) implies that $t_1,t_2\in \mathcal{E}_1$. Thus, due to the fact that $R_{I,\omega}(t_1)\ne R_{I,\omega}(t_2)$, we obtain (a).

Next, let condition (ii) hold. Then, as noted at the beginning of the proof, for $\gamma={\rm{arg}} \tilde{I}({\omega}(t_0))$ there exists a point $t_1\in [0,2\pi]$ for which \eqref{main weight 2} holds. Thus, $t_1\in\Phi$ by the definition of this set, and hence from condition (ii) it follows that $t_1\in R_{I,\omega}(t_0)$. This means that ${\rm{arg}} \tilde{I}({\omega}(t_0))={\rm{arg}} \tilde{I}({\omega}(t_1))+\pi k$ for some $k\in\mathbb{Z}$.
Thus, by elementary properties of the sine function and lemma \ref{lemma: calcul}, we obtain the following:
\begin{eqnarray*}
\xi_{I,\omega}^2(u,v,\gamma)&=&|\sin(({\rm{arg}} \tilde{I}({\omega}(v))+{\rm{arg}} \tilde{I}({\omega}(u)))/2-{\rm{arg}} \tilde{I}({\omega}(t_0)))|\\&=&
|\sin(({\rm{arg}} \tilde{I}({\omega}(v))+{\rm{arg}} \tilde{I}({\omega}(u)))/2-{\rm{arg}} \tilde{I}({\omega}(t_1)))|
\\&\le& \big|\sin\big(({\rm{arg}} \tilde{I}({\omega}(u))-{\rm{arg}} \tilde{I}({\omega}(t_1)))/2\big)\big|\\&+& \big|\sin(({\rm{arg}} \tilde{I}({\omega}(v))-{\rm{arg}} \tilde{I}({\omega}(t_1)))/2\big)\big|\\&\le&
2C_a\max(|{\omega}(u)-{\omega}(t_1)|,|{\omega}(v)-{\omega}(t_1)|).
\end{eqnarray*}
Hence, comparing relations \eqref{main weight 2} (with $\gamma={\rm{arg}} \tilde{I}({\omega}(t_0))$, $t=t_1$) and \eqref{ineq5b_gen}, as well as taking into account inequality \eqref{main weight 2a}, we conclude that $t_1\in \mathcal{E}_2$. Thus, $\mathcal{E}_2\ne\emptyset$, i.e., we have (b). This completes the proof.

\end{proof}

We will refer to points from the set $\mathcal{E}_1(\mu,\omega)$ as to "critical"\:ones for the function $\mu^*$. The following important consequence of the last theorem for non outer functions indicates that if the function $\mu^*$ has at most one "critical"\:point and $\mathcal{E}_2(\mu,\omega)=\emptyset$, then regardless of inner factors in representation \eqref{genl3repr}, the function $f$ is not an extreme point of the ${\rm ball}\,(H(\Lambda(\varphi)))$.

\begin{corollary}\label{ext of B-S-general}
Assume that the conditions of Theorem \ref{Th4a general} are satisfied.
Then, if $f$ is a non outer function, ${\rm card}\,\mathcal{E}_1(\mu,\omega)\le 1$, $\mathcal{E}_2(\mu,\omega)=\emptyset$, then $f$ is not an extrem point of the ${\rm ball}\,(H(\Lambda(\varphi)))$.
\end{corollary}
\begin{proof}
Since $f$ is not an outer function, there are a nonconstant inner function $I$ and $F\in H^1$ such that $f=F\cdot I$. From the assumption that ${\rm card}\,\mathcal{E}_1(\mu,\omega)\le 1$ and the fact that $\bigcup_{t\in [0,2\pi]} R_{I,\omega}(t)=[0,2\pi]$ it follows that $\mathcal{E}_1(\mu,\omega)\subset R_{I,\omega}(t_0)$ for some $t_0\in [0,2\pi]$. Therefore, since $\mathcal{E}_2(\mu,\omega)=\emptyset$, by Theorem \ref{Th4a general}, we get the desired result.

\end{proof}


Next, we proceed with considering the more specialized case when the function $\mu:=|\tilde{f}|$ decreases on $[0,2\pi]$. 

Without loss of generality, we can assume that $\mu$ is continuous on the left at each point of the interval $[0,2\pi]$, and hence $\mu(t)=\mu^*(t)$ for all $0\le t\le 2\pi$ \cite[\S\,II.2]{krein-petunin-semenov}. Since in this case $\omega_0(t)=t$, $0\le t\le 2\pi$, 
the sets $\mathcal{E}_1(\mu,\omega_0)$ and $\mathcal{E}_2(\mu,\omega_0)$ (which will be denoted by $\mathcal{E}_1(\mu)$ and $\mathcal{E}_2(\mu)$) consist of all $t\in [0,2\pi]$ such that
\begin{equation} \label{ineq5b1}
\liminf_{u\to t,v\to t,0<u<v\le 2\pi}\frac{\mu(u)-\mu(v)}{\mu(v)(v-u)}=0
\end{equation}
and
\begin{equation} \label{ineq5b}
\liminf_{u\to t,v\to t,0<u<v\le 2\pi}\frac{\mu(u)-\mu(v)}{(v-u)\mu(v)\max(|u-t|,|v-t|)}=0,
\end{equation}
respectively. Similarly, the sets $R_{I,\omega_0}(t)$ will be denoted by $R_{I}(t)$.

The following result is an immediate consequence of Theorem \ref{Th4a general}.

\begin{corollary}\label{ext of B-S}
Let $\varphi$ be a strictly increasing and strictly concave function on $[0,2\pi]$.
Suppose that $\|f\|_{H(\Lambda(\varphi))}=1$ and $\mu:=|\tilde{f}|$ is a decreasing function on $[0,2\pi]$.

If there are an inner function $I\in \mathcal{I}(f)$ and $t_0\in [0,2\pi]$ such that $\mathcal{E}_1(\mu)\subset R_{I}(t_0)$ and $\mathcal{E}_2(\mu)=\emptyset$, then $f$ is not an extreme point of the unit ball in the space $H(\Lambda(\varphi))$.
\end{corollary}

\begin{remark}\label{ext of B-S result}
One can easily see that the assumptions of Theorem D by Bryskin-Sedaev from Subsect. \ref{bryskin-sedaev} ensure that $\mathcal{E}_1(\mu)=\emptyset$. Therefore, the latter theorem is a very special consequence of Corollary \ref{ext of B-S}. 
\end{remark}

We are able to obtain more precise results in the case when the function $f$ is the product of an outer function and a Blaschke factor. 

In the case when $\omega_0(t)=t$ and $I$ is the Blaschke factor $I_a$, $a\in\D$ (see formula \eqref{factor}), we will adopt the following notation:
$$
\xi_a^1(u,v):=\xi_{I_a,\omega_0}^1(u,v)= |\sin(({\rm{arg}} \tilde{I_a}(v)-{\rm{arg}} \tilde{I_a}(u))/2)|,$$
$$
\xi_a^2(u,v,\gamma):=\xi_{I_a,\omega_0}^2(u,v)=|(\sin(({\rm{arg}} \tilde{I_a}(v)+{\rm{arg}} \tilde{I_a}(u))/2-\gamma)|
$$
(see \eqref{quant1} and \eqref{quant2}) and
$$
R_a(t):=R_{I_a}(t),\;\;0\le t<2\pi.$$
Recall that for every $a\in\D$ and $t\in [0,2\pi]$ the set $R_{a}(t)$ consists of all $s\in [0,2\pi]$ such that the difference ${\rm{arg}}\tilde{I}_a(s)-{\rm{arg}}\tilde{I}_a(t)$ is a multiple of $\pi$.

\begin{theorem}\label{Th4a}
Let $\varphi$ be a strictly increasing and strictly concave function on $[0,2\pi]$. 
Suppose that  $f=F\cdot I_a$, where $a\in \mathbb{D}$ and $F$ is an outer analytic function  such that $\|F\|_{H(\Lambda(\varphi))}=1$ and the function $\mu:=|\tilde{F}|$ decreases on the interval $[0,2\pi]$.

Then $f$ is an extreme point of the unit ball in $H(\Lambda(\varphi))$ if and only if at least one of the following two conditions is satisfied:

(a) there is no $t_0\in [0,2\pi]$ such that $\mathcal{E}_1(\mu)\subset R_{a}(t_0)$;

(b) $\mathcal{E}_2(\mu)\ne\emptyset$.
\end{theorem}
\begin{proof}
By Theorems \ref{cor Th4a new} (see the second assertion) and
\ref{Th4a general}, it suffices to show that each of conditions (a) or (b) implies that for any $\gamma\in \mathbb{R}$ there exists a point $t\in [0,2\pi]$, for which we have
\begin{equation} \label{main weight 2new}
\liminf_{u\to t,v\to t,0<u<v\le 2\pi}\frac{\mu(u)-\mu(v)}{\mu(v)\xi_a^1(u,v)\xi_a^2(u,v,\gamma)}=0.
\end{equation}

We assume first that (a) holds, i.e., there are $t_1,t_2\in \mathcal{E}_1(\mu)$   such that $R_a(t_1)\ne R_a(t_2)$. Observe that by Lemma \ref{lemma: calcul},  for all sufficiently close $u$ and $v$, $u<v$, we have
\begin{equation} \label{lower ineq5b12}
\xi_a^1(u,v)\ge c_a(v-u).
\end{equation}
In consequence, from definition of the set $\mathcal{E}_1(\mu)$ (see \eqref{ineq5b1}) it follows that for $i=1,2$
\begin{equation} \label{ineq5b12}
\liminf_{u\to t_i,v\to t_i,0<u<v\le 2\pi}\frac{\mu(u)-\mu(v)}{\mu(v)\xi_a^1(u,v)}=0.
\end{equation}

Further, since the function $s\mapsto {\rm{arg}}\tilde{I_a}(s)$ is continuous on $\mathbb{R}$, we have
$$
\lim_{u,v\to t_i}\xi_a^2(u,v,\gamma)=|\sin({\rm{arg}} \tilde{I_a}(t_i)-\gamma)|,\;\;i=1,2,$$
for any $\gamma$. Moreover, from the assumption $R_a(t_1)\ne R_a(t_2)$ it follows that ${\rm{arg}} \tilde{I_a}(t_1)\ne{\rm{arg}} \tilde{I_a}(t_2)+\pi k$ for any $k\in\mathbb{Z}$. Hence, for each $\gamma\in\mathbb{R}$ at least one of the limits $\lim_{u,v\to t_1}\xi_a^2(u,v,\gamma)$ and $\lim_{u,v\to t_2}\xi_a^2(u,v,\gamma)$ does not vanish. Therefore, assuming, say, that $\lim_{u,v\to t_1}\xi_a^2(u,v,\gamma)\ne 0$, we get equality \eqref{main weight 2new} for this $\gamma$ and $t=t_1$ as a consequence of  equality \eqref{ineq5b12} for $i=1$.

Now let condition (b) hold. Then, there is $t\in \mathcal{E}_2(\mu)$, and hence, by  \eqref{ineq5b}, for some sequences $\{u_n\}_{n=1}^\infty$ and $\{v_n\}_{n=1}^\infty$ satisfying the conditions: $0<u_n<v_n\le 2\pi$, $n=1,2,\dots$, $\lim_{n\to\infty}u_n=\lim_{n\to\infty}v_n=t$, we have
\begin{equation}
\label{final case}
\lim_{n\to\infty}\frac{\mu(u_n)-\mu(v_n)}{(v_n-u_n)\mu(v_n)\max(|u_n-t|,|v_n-t|)}=0.
\end{equation}
If $\gamma\ne {\rm{arg}} \tilde{I_a}(t)+\pi k$ for any integer $k$, it holds
$$
\lim_{n\to\infty}\xi_a^2(u_n,v_n,\gamma)=|\sin({\rm{arg}} \tilde{I_a}(t)-\gamma)|
\ne 0,$$ 
and hence \eqref{main weight 2new} is an immediate consequence of \eqref{final case} and inequality \eqref{lower ineq5b12}. 

Thus, it remains to prove \eqref{main weight 2new} in the case when $\gamma= {\rm{arg}} \tilde{I_a}(t)+\pi k$ for some $k\in\mathbb{Z}$.
Passing to subsequences, we can assume that one of the following three conditions is satisfied: (i) $u_n\le t\le v_n$, $n=1,2,\dots$; (ii) $u_n<v_n<t$, $n=1,2,\dots$; (iii) $t<u_n<v_n$, $n=1,2,\dots$.

Assume that (i) holds. Since $u_n>0$, then in this case $t>0$. If necessary, by passing to further subsequences, we will have that $\max(|u_n-t|,|v_n-t|)=v_n-t$ or $\max(|u_n-t|,|v_n-t|)=t-u_n$, $n=1,2,\dots$. Let, say, the first inequality hold. Because $\mu(t)-\mu(v_n)\le \mu(u_n)-\mu(v_n)$ and $v_n-u_n\le 2(v_n-t)$,
from \eqref{final case} it follows that
\begin{equation}
\label{final case1}
\lim_{n\to\infty}\frac{\mu(t)-\mu(v_n)}{\mu(v_n)(v_n-t)^2}=0.
\end{equation}
Moreover, by Lemma \ref{lemma: calcul}, for sufficiently large $n$
$$
\xi_a^1(t,v_n)\ge c_a(v_n-t)$$
and
\begin{eqnarray*}
\xi_a^2(t,v_n,\gamma)&=&|\sin({\rm{arg}} \tilde{I_a}(v_n)+{\rm{arg}} \tilde{I_a}(t))/2-{\rm{arg}} \tilde{I_a}(t))|\\&=&|\sin(({\rm{arg}} \tilde{I_a}(v_n)-{\rm{arg}}\tilde{I_a}(t))/2)|\\&\ge& c_a(v_n-t).
\end{eqnarray*}
Therefore, for such $n$ we have
$$
\frac{\mu(t)-\mu(v_n)}{\mu(v_n)\xi_a^1(t,v_n)\xi_a^2(t,v_n,\gamma)}\le c_a^{-2}\frac{\mu(t)-\mu(v_n)}{\mu(v_n)(v_n-t)^2},$$
and hence \eqref{main weight 2new} follows from \eqref{final case1}.

If $\max(|u_n-t|,|v_n-t|)=t-u_n$, $n=1,2,\dots$, then exactly in the same way, as above, from the inequalities
$\mu(u_n)-\mu(t)\le \mu(u_n)-\mu(v_n)$, $v_n-u_n\le 2(t-u_n)$, \eqref{final case} and the fact that $\lim_{n\to\infty}\mu(v_n)\le\mu(t)<\infty$, it follows that
$$
\lim_{n\to\infty}\frac{\mu(u_n)-\mu(t)}{(t-u_n)^2}=0.
$$
Furthermore, again for sufficiently large $n$
$$
\xi_a^1(u_n,t,\gamma)\xi_a^2(u_n,t,\gamma)\ge c_a^2(u_n-t)^2,$$
whence
$$
\frac{\mu(u_n)-\mu(t)}{\mu(t)\xi_a^1(u_n,t)\xi_a^2(u_n,t,\gamma)}\le c_a^{-2}\cdot \frac{\mu(u_n)-\mu(t)}{\mu(t)(t-u_n)^2}.$$
Combining this together with the last equality, we obtain \eqref{main weight 2new}.

Let us proceed with the case (ii). As was already noted, the function 
$t\mapsto{\rm{arg}} \tilde{I_a}(t)$ is a continuous and one-to-one mapping from $\R$ onto $\R$. Therefore, for any real $u$ and $v$ such that $u<v<t$, one of two following inequalities holds:
\begin{equation}
\label{final alternative}
{\rm{arg}} \tilde{I_a}(u)<{\rm{arg}} \tilde{I_a}(v)<{\rm{arg}} \tilde{I_a}(t)\;\;\mbox{or}\;\;{\rm{arg}} \tilde{I_a}(u)>{\rm{arg}} \tilde{I_a}(v)>{\rm{arg}} \tilde{I_a}(t).
\end{equation}
Assuming, for instance, that the first of them is satisfied, we obtain
$$
{\rm{arg}} \tilde{I_a}(u_n)<{\rm{arg}} \tilde{I_a}(v_n)<{\rm{arg}} \tilde{I_a}(t), \;\;n=1,2,\dots,$$
whence for sufficiently large $n$, by Lemma \ref{lemma: calcul},
\begin{equation}
\label{final lower est1}
\xi_a^1(u_n,v_n)\ge c_a(v_n-u_n)
\end{equation}
and
\begin{eqnarray}
\xi_a^2(u_n,v_n,\gamma)&=&|\sin({\rm{arg}} \tilde{I_a}(u_n)+{\rm{arg}} \tilde{I_a}(v_n))/2-{\rm{arg}} \tilde{I_a}(t))|\nonumber\\&=&\sin({\rm{arg}} \tilde{I_a}(t)-({\rm{arg}} \tilde{I_a}(u_n)+{\rm{arg}} \tilde{I_a}(v_n))/2))\nonumber\\&\ge&\sin(({\rm{arg}} \tilde{I_a}(t)-{\rm{arg}}\tilde{I_a}(u_n))/2)\nonumber\\&\ge& c_a(t-u_n)=c_a\max(t-v_n,t-u_n).
\label{final lower est2}
\end{eqnarray}
Thus, for such $n$ we have
$$
\frac{\mu(u_n)-\mu(v_n)}{\mu(v_n)\xi_a^1(u_n,v_n)\xi_a^2(u_n,v_n,\gamma)}\le c_a^{-2}\frac{\mu(u_n)-\mu(v_n)}{(v_n-u_n)\mu(v_n)\max(|u_n-t|,|v_n-t|)},
$$
and now \eqref{main weight 2new} follows from \eqref{final case}.

The proof of estimates \eqref{final lower est1} and \eqref{final lower est2}
in the case when the second inequality in \eqref{final alternative} is satisfied and also  in the case (iii) can be carried in the same way, and hence we skip this. Thus, since \eqref{main weight 2new} is an immediate consequence of \eqref{final case}, \eqref{final lower est1}, and \eqref{final lower est2}, the proof is completed.


\end{proof}

Since the function $\tilde{I}_a:\,[0,2\pi)\to \T$ is a one-to-one mapping and $\tilde{I}_a(2\pi)=\tilde{I}_a(0)$, then each of the sets $R_{a,\omega}(t)$ consists of two or three points. In particular, ${\rm card}\,R_{a,\omega}(t)=3$ if and only if $0\in R_{a,\omega}(t)$.
Therefore, applying Theorem \ref{Th4a} and well-known properties of the linear-fractional automorphisms of the unit disk, we obtain the following statements, which show that geometric properties of a function $f$ of the form \eqref{partial}, such that the function $|\tilde{F}|$ decreases, depend, first of all, on the number and the structure of "critical"\:points of $|\tilde{F}|$, and, possibly, secondarily, on the choice of the point $a\in\D$. 

\begin{corollary}\label{cor for Th4a}
Assume that all the conditions of Theorem \ref{Th4a} are fulfilled.

Then, if $f_a(z):=F(z)I_a(z)$, where $a\in\mathcal{D}$, and $\mu:=|\tilde{F}|$, we have the following:

(i) if ${\rm card}\,\mathcal{E}_1(\mu)\ge 4$ or $\mathcal{E}_2(\mu)\ne \emptyset$, then $f_a$ is an extreme point of the unit ball of $H(\Lambda(\varphi))$ for any $a\in\mathcal{D}$;

(ii) if ${\rm card}\,\mathcal{E}_1(\mu)\le 1$ and $\mathcal{E}_2(\mu)=\emptyset$, then $f_a$ is not an extreme point of the unit ball of $H(\Lambda(\varphi))$ for any $a\in\mathcal{D}$;

(iii) 
if ${\rm card}\,\mathcal{E}_1(\mu)=2$ and
$\mathcal{E}_2(\mu)=\emptyset$, then there exist $a_j\in\mathcal{D}$, $j=1,2$, such that $f_{a_1}$ is an extreme point of the unit ball of $H(\Lambda(\varphi))$, but $f_{a_2}$ is not.

\end{corollary}

Observe that the set of outer analytic functions that satisfy each one of the conditions (i), (ii) and (iii), is non-empty.
Indeed, for any symmetric space $X$ on $[0,2\pi]$, each $\mu\in X$ such that $\mu\ge 0$ and $\ln\mu\in L^1[0,2\pi]$, the analytic function
$$
F(z):=\exp\left\{\frac1{2\pi}\int_0^{2\pi}\frac{e^{it}+z}{e^{it}-z}\ln\mu(t)\,dt+i\lambda\right\},\;\;z\in\D,$$
where $\lambda\in\mathbb{R}$, is outer, $|\tilde{F}|=\mu$ (see, e.g., \cite[Chapter 5]{hoffman}), and hence $\tilde{F}\in X$.

\begin{remark}
Recall that $\lambda\in [-\infty,\infty]$ is a {\it derivative number} of a function ${\mu}:\,[0,T]\to\mathbb{R}$ at $t_0\in[0,T]$ if for some sequence $t_k\to t_0$ for $k\to\infty$
$$
\lim_{k\to\infty}\frac{{\mu}(t_k)-{\mu}(t_0)}{t_k-t_0}=\lambda.$$
It is clear that if $0$ is a derivative number of a decreasing positive function $\mu$ at $t_0$, it follows that $t_0\in \mathcal{E}_1({\mu}).$ The converse is generally not true: there exists a strictly decreasing function ${\mu}$ such that all its derivative numbers at each point of the interval $[0,T]$ are negative, but however $\mathcal{E}_1({\mu})$ contains as many elements as desired. In particular, in this case the function $f_a$ from the last corollary may or may not be an extreme point of the unit ball of the space $H(\Lambda(\varphi))$ (depending on what the sets $\mathcal{E}_1(\mu)$, $\mathcal{E}_2(\mu)$ and the point $a\in\mathcal{D}$).
\end{remark}

\vspace{3em}
The author is grateful to A.D. Baranov and I.R. Kayumov for useful discussions of the issues considered in this paper.

\newpage



\end{document}